\documentclass{amsart}

\usepackage{amsmath,amsfonts,amssymb,amsthm,hyperref}
\usepackage{color}
\usepackage{centernot}

\newtheorem{thm}{Theorem}[section]

\newtheorem{cor}[thm]{Corollary}

\newtheorem{lem}[thm]{Lemma}
\newtheorem{lemma}[thm]{Lemma}

\newtheorem{prop}[thm]{Proposition}

\newtheorem{remark}[thm]{Remark}
\newtheorem{defn}[thm]{Definition}

\newtheorem{question}[thm]{Question}

\newtheorem{example}[thm]{Example}

\newcommand{\Q}{\mathbb{Q}}

\newcommand{\Z}{\mathbb{Z}}
\newcommand{\N}{\mathbb{N}}

\newcommand{\rk}{\operatorname{rank}}

\newcommand{\dom}[1]{\text{dom}(#1)}

\newcommand{\tfa}{\operatorname{TFAb}}
\newcommand{\TFAb}{\operatorname{TFAb}}
\newcommand{\FD}{\operatorname{FD}}
\newcommand{\TD}{\operatorname{FD}}
\newcommand{\set}[2]{\ensuremath{ \{ #1 : #2 \} }}
\newcommand{\bfd}{\boldsymbol{d}}

\newcommand{\xvec}{\vec{x}}

\newcommand{\A}{\mathcal{A}}
\newcommand{\B}{\mathcal{B}}
\newcommand{\C}{\mathcal{C}}

\newcommand{\Ftilde}{\widetilde{F}}

\newcommand{\Gtilde}{\widetilde{G}}
\newcommand{\Htilde}{\widetilde{H}}

\newcommand{\Xtilde}{\tilde{X}}

\newcommand{\la}{\langle}
\newcommand{\ra}{\rangle}
\newcommand{\Iso}[1]{\text{Iso}(#1)}

\newcommand{\comment}[1]{}

\newenvironment{pf}{\begin{trivlist}\item[\hskip\labelsep
{\it Proof.}]}{\end{trivlist}}

\date{\today}
\begin{document} 

\title[$\tfa_r$ and $\TD_r$]{Torsion-free abelian groups of finite rank and fields of finite transcendence degree}

\author{Meng-Che ``Turbo'' Ho \and Julia Knight \and Russell Miller}
\thanks{This material is based on work supported by the National Science Foundation under Grant \#DMS-1928930 while the authors participated in programs hosted by the Mathematical Sciences Research Institute in Berkeley, California, during Fall 2020 and Summer 2022.}
\thanks{The first author acknowledges support from the National Science Foundation under Grant \#DMS-2054558.}
\thanks{The second author acknowledges support from the National Science Foundation under Grant \#DMS-1800692.}
\thanks{The third author acknowledges support from the Simons Foundation under Grant \#581896,
and from The City University of New York PSC-CUNY Research Award Program under under Grant \#66379-00 54.}

\maketitle 

\begin{abstract}

Let $\TFAb_r$ be the class of torsion-free abelian groups of rank $r$, and let $\TD_r$ be the class of fields of characteristic $0$ and transcendence degree~$r$.  
We compare these classes using various notions. Considering Scott complexity of the structures in the classes and the complexity of the isomorphism relations on the classes, the classes seem very similar. Hjorth and Thomas showed that the $\TFAb_r$ are strictly increasing under Borel reducibility.  This is not so for the classes $\TD_r$.  Thomas and Velickovic showed that for sufficiently large $r$, the classes $\TD_r$ are equivalent under Borel reducibility.  We try to compare the groups with the fields, using Borel reducibility, and also using some effective variants.  We give functorial Turing computable embeddings of $\TFAb_r$ in $\TD_r$, and of $\TD_r$ in $\TD_{r+1}$.  We show that under computable countable reducibility, $\TFAb_1$ lies on top among the classes we are considering.  In fact, under computable countable reducibility, isomorphism on $\TFAb_1$ lies on top among equivalence relations that are effective $\Sigma_3$, along with the Vitali equivalence relation on $2^\omega$.  
\end{abstract}         

\section{Introduction}
\label{sec:intro}

There are substantial similarities between the class $\TFAb$ of torsion-free abelian groups of finite rank
and the class $\TD$ of fields of characteristic $0$ having finite transcendence degree over $\Q$.
Both of these well-studied classes consist of countable structures. Except for the trivial group, which we ignore, all are infinite. Hence, we may suppose that the universe of each structure is $\omega$.  For each class, there is a dependence notion such that the size of a maximal independent set or \emph{basis} is well-defined.  Each structure is determined, up to isomorphism, by the existential type of a basis. The existential type of the basis says which rational linear combinations are present (in the group), or which polynomials have roots (in the field). 

We are particularly interested in the subclasses of $\TFAb$ and $\TD$ for which the size of a basis is fixed.  We write $\TFAb_r$ for the class of torsion-free abelian groups of rank $r$, and $\TD_r$ for the class of fields of characteristic $0$ and transcendence degree~$r$.  Each of these sub-classes has a universal structure.  The elements of $\TFAb_r$ are precisely the isomorphic copies of the rank $r$ subgroups of the additive group $\Q^r$.  The elements of $\TD_r$ are precisely the isomorphic copies of degree $r$ subfields of $\overline{\Q(t_1,\ldots,t_r)}$, where $\Q(t_1,\ldots,t_r)$ is the purely transcendental extension of $\Q$ of degree $r$, and $\overline{Q(t_1,\ldots,t_r)}$ is its algebraic closure.     

The class $\TFAb$ and the class $\TD$ share many computable-structure-theoretic properties.  With one exception, the degree spectra for structures in the two classes are the same---they are the degrees of sets $C$ that are c.e.\ relative to some fixed $d$.   The lone exception is the 
trivial group $\{0\}$, which is the unique group in $\TFAb_0$; for fields, in contrast,
$\TD_0$ contains uncountably many fields, all of which follow the rule above. 
The proposition below gives the complexities of the isomorphism relations on the subclasses $\TFAb_r$ and $\TD_r$.  We say only a little about the proof. 

\begin{prop}\

\begin{enumerate}

\item  For $\TFAb_0$, the isomorphism relation is trivial.

\item  For $\TD_0$, the isomorphism relation is effective $\Pi_2$.

\item  For $\TFAb_r$ and $\TD_r$, where $r \ge 1$, the isomorphism relation is effective $\Sigma_3$.

\end{enumerate}

\end{prop}

\begin{proof} [Partial proof]

(2)  Fields in $\FD_0$ are isomorphic if and only if they satisfy the same existential sentences.  This is effective $\Pi_2$.  

(3)  For $\TFAb_r$ and $\FD_r$ for $r > 0$, the isomorphism relation is defined by a computable $\Sigma_3$ formula saying that there are bases of size $r$ for the two structures such that the existential formulas true of the bases are the same.  
\end{proof} 

The results above are sharp.  In \cite[Cor.\ 2.8]{HKMS}, it is shown that the set of pairs of indices for isomorphic computable structures in $\FD_0$ is $\Pi^0_2$.  (To see that this is $1$-complete at this level, just notice that
when $\{ W_e\}_{e\in\omega}$ is the usual effective listing of c.e.\ sets and $p_0<p_1<\cdots$ are the primes,
$W_i=W_j$ just if $\Q(\sqrt{p_n}~:~n\in W_i) \cong \Q(\sqrt{p_n}~:~n\in W_j)$.)
In Section 6, we will show more.  The isomorphism relation on $\FD_0$ is complete effective $\Pi_2$, and complete $\mathbf{\Pi_2}$, under reducibilities appropriate for the effective Borel, and Borel hierarchies.  Similarly, for the classes $\TFAb_r$ and $\FD_r$ for $r > 0$, we will show that the isomorphism relation is complete effective $\Sigma_3$, and complete $\mathbf{\Sigma_3}$.  

\bigskip 

In her PhD thesis, Alvir \cite{Alvir} generalized the notion of finitely generated structure.

\begin{defn} [Alvir]
\label{defn:alphafingen}

A structure $\A$ is \emph{$\alpha$-finitely generated} if there is a finite tuple $\bar{a}$ such that
for all tuples $\bar{b}$ from $\A^{<\omega}$, the orbit of $\bar{b}$ over $\bar{a}$ is defined by an
infinitary $\Sigma_\alpha$ formula.  

\end{defn}

In the classes $\TFAb_r$ and $\FD_r$, all structures are $1$-finitely generated,
with any basis serving as $\bar a$.  The orbit of a tuple $\bar b$ \emph{over the basis $\bar{a}$} is defined by an existential formula.  For the groups in $\TFAb$, each element $b$ is actually defined
over a basis $\bar a$ by a quantifier-free formula of the form $n\cdot b = \sum m_ia_i$
with integer coefficients $n,m_i$.  In $\TD$, however, existential quantifiers are required,
and finitely many distinct tuples $\bar b$ can realize the same existential type over $\bar a$.
In both classes, viewing the structures as subgroups of $\Q^r$ or subfields
of $\overline{\Q(t_1,\ldots,t_d)}$ endows them (as a class) with computability-theoretic
properties different from those they possess as free-standing structures.  In particular, the relations of linear independence (for the groups) and algebraic independence (for the fields) are uniformly decidable.  In the setting of free-standing structures (which we use), independence is decidable
from the atomic diagram of the structure, but not uniformly so, as one needs
to know a basis for the structure.  (The exceptions are
$\TFAb_0$ and $\TD_0$, where every basis is empty,
and $\TFAb_1$, where every nonzero singleton is a basis).

There are differences among classes $\TFAb_r$ for different $r$ that are not accounted for by the complexity of the Scott sentences or the isomorphism relation.  There are ``invariants'' for $\TFAb_1$ that are widely accepted as useful.  This is not the case for $\TFAb_r$ for $r > 1$.  Hjorth \cite{H} and Thomas \cite{T} used the notion of Borel embedding (see Definition \ref{defn:Borelembedding} below) to say in a precise way that the complexity of the invariants increases with $r$.

\begin{thm} [Hjorth, Thomas]
\label{thm:HjorthThomas}

For each $r\geq 1$, $\TFAb_r <_B \TFAb_{r+1}$.

\end{thm}  

The fact that $\TFAb_r \leq_B\TFAb_{r+1}$ simply says that there is a Borel function $F$
that, given the atomic diagram of any $G\in\TFAb_{r+1}$, produces the atomic diagram
of some group $F(G)\in\TFAb_r$, in such a way that $G_0\cong G_1$ if and only if
$F(G_0)\cong F(G_1)$.  That is, the isomorphism problem for $\TFAb_{r}$ is
\emph{Borel-reducible} to that for $\TFAb_{r+1}$.  This is not surprising; indeed,
a straightforward computable function $F(G) = G\times \Z$ can accomplish this task.  However,
the result of Hjorth and Thomas gives strict Borel reducibility
$\TFAb_r <_B \TFAb_{r+1}$, meaning that there is no Borel reduction in the
opposite direction:  $\TFAb_{r+1} \not\leq_B\TFAb_r$.  The proof uses deep results from descriptive set theory.  
Hjorth and Thomas mentioned the case of fields of finite transcendence degree, but did not
address it to any significant extent.  Thomas and Velickovic \cite{TV} have shown that the classes $\TD_n$ are not strictly increasing under Borel reducibility.  There is some (fairly small) $n$ such that for all $m$, $\TD_m\leq_B \TD_n$.      

Our purpose in this article is both to consider the parallel questions for the different ranks $\TD_r$ of fields of finite transcendence degree, and to apply the notions of computable reducibility that were
subsequently developed in \cite{CCKM} and \cite{MPSS18}.  We will show
that for each $r$, there is a Turing computable reduction from $\FD_r$ to $\FD_{r+1}$.  (This is not nearly so simple as it was for $\TFAb$, as we discuss in Section \ref{sec:fieldsup}.)
We will also give, in Section \ref{sec:groupstofields}, a Turing computable reduction
from each $\TFAb_r$ to the corresponding $\TD_r$.  All of these reductions
will in fact be \emph{functorial}, a particularly strong type of Turing computable 
reduction
that we describe in Section \ref{sec:reducibility} after giving history and technical details
about Borel and Turing computable reductions.

By results of Hjorth \cite{H} and Thomas \cite{T}, there is no Borel reduction from $\TFAb_r$ to $\TFAb_{r'}$ for $r > r'$.  We do not know a specific Borel reduction from $\TD_r$ to $\TD_{r'}$ for $r > r'$, but results of Thomas and Velickovic imply that for all sufficiently large $r$, the classes $\TD_r$ are $\equiv_B$-equivalent. 
While the previously mentioned results imply that for large enough $d$, there is no Borel reduction from $\FD_d$ to $\TFAb_r$, we do not yet know a Borel reduction from any $\TD_d$ with $d > 0$ to any $\TFAb_r$.
However, we will show that under \emph{countable computable reducibility}, all effective $\Sigma_3$ equivalence relations on $2^\omega$ reduce to $\TFAb_r$.  Isomorphism on 
$\TFAb_r$, for $r\geq 1$, and $\TD_r$ are effective $\Sigma_3$. The notion of countable computable reducibility was introduced in~\cite{MSEALS}.  






\subsection*{Acknowledgment}
The authors would like to thank Phillip Dittman, Matthew Harrison-Trainor, Vincent Ye, and Arno Fehm for a useful discussion that 
led to the proof in Section \ref{sec:fieldsup}.

\section{Borel and Turing computable reducibility}
\label{sec:reducibility}

For a countable language $L$, $Mod(L)$ is the set of all $L$-structures with universe~$\omega$.  For convenience, we may suppose that $L$ consists of relation symbols.  We can identify $Mod(L)$ with $2^\omega$.  Let $(\alpha_n)_{n\in\omega}$ be an enumeration of the sentences of form $R(\bar{a})$, where $R$ is a relation symbol of $L$ and $\bar{a}$ is an appropriate tuple of natural numbers.  We identify the structure $\mathcal{A}\in Mod(L)$ with the function $f\in 2^\omega$ such that
$$f(n) = \left\{\begin{array}{ll}
1 & \mbox{if $\mathcal{A}\models\alpha_n$}\\
0 & \mbox{otherwise.}
\end{array}\right.$$

We have the usual topology on $2^\omega$, and on $Mod(L)$.  The basic clopen sets in $Mod(L)$ have the form $Mod(\varphi) = \set{\mathcal{A}\in Mod(L)}{\mathcal{A}\models \varphi}$, where $\varphi$ is finitary and quantifier-free, in the language $L$ with added constants for the natural numbers.  The Borel sets are the members of the $\sigma$-algebra generated by the basic clopen sets.  The Borel sets may be obtained from the basic clopen sets by closing under countable unions and intersections.  The \emph{Borel hierarchy} classifies Borel sets as $\mathbf{\Sigma_\alpha}$ or $\mathbf{\Pi_\alpha}$ for countable ordinals $\alpha$.  The \emph{effective Borel sets} are obtained from the basic clopen sets by taking c.e.\ unions and intersections.  The \emph{effective Borel hierarchy} classifies sets as $\Sigma_\alpha$ or $\Pi_\alpha$ for computable ordinals $\alpha$.  Recall that the $L_{\omega_1\omega}$ formulas allow countably infinite disjunctions and conjunctions but only finite strings of quantifiers.  The formulas are classified as $\Sigma_\alpha$ or $\Pi_\alpha$ for countable ordinals $\alpha$.
The \emph{computable infinitary formulas} are formulas of $L_{\omega_1\omega}$ in which the disjunctions and conjunctions are over c.e.\ sets.            

Fixing a language $L$, we consider classes $K\subseteq Mod(L)$ such that $K$ is closed under isomorphism; i.e., $K$ is closed under the action of the permutation group $S_\infty$ on $\omega$.  Lopez-Escobar \cite{Lopez-Escobar} showed that such a class is Borel if and only if it is axiomatized by a sentence of $L_{\omega_1\omega}$.    
Vaught \cite{Vaught} showed that for any countable ordinal $\alpha\geq 1$ and any class $K\subseteq Mod(L)$ (closed under isomorphism), $K$ is $\mathbf{\Sigma_\alpha}$ if and only if it is axiomatized by a $\Sigma_\alpha$ sentence of $L_{\omega_1\omega}$.  Vanden Boom \cite{VdB} proved the effective version of Vaught's Theorem, saying that for any computable ordinal $\alpha\geq 1$, a class $K\subseteq Mod(L)$ (closed under isomorphism) is effective $\Sigma_\alpha$ if and only if it is axiomatized by a computable $\Sigma_\alpha$ sentence.    

We have the usual product topology on $2^{\omega}\times 2^{\omega}$, and on $Mod(L)\times Mod(L')$,
so we may consider Borel relations on $Mod(L)\times Mod(L')$ and Borel functions from $Mod(L)$ to $Mod(L')$.  
Friedman and Stanley \cite{FS} introduced the notion of Borel embedding as a precise way
to compare the problems of classifying, up to isomorphism, the members of different classes of countable structures.  

\begin{defn} [Friedman--Stanley]
\label{defn:Borelembedding}
Suppose $K\subseteq Mod(L)$, $K'\subseteq Mod(L')$ are closed under isomorphism.  A \emph{Borel embedding} of $K$ in $K'$ is a Borel function $\Phi:K\rightarrow K'$ such that for $\mathcal{A},\mathcal{B}\in K$, $\mathcal{A}\cong\mathcal{B}$ if and only if $\Phi(\mathcal{A})\cong\Phi(\mathcal{B})$.  We say that class $K$ is \emph{Borel embeddable in} the class $K'$, and we write \emph{$K \leq_B K'$}, if there is such an embedding.        

\end{defn}

Friedman and Stanley gave a number of examples involving familiar classes of structures.  In particular, they showed that fields and linear orderings lie on top under $\leq_B$.  (This means that each of these classes is maximal under $\leq_B$ among Borel classes of $L$-structures for every countable language $L$.)  Using known results, they obtained the fact that $2$-step nilpotent groups lie on top, but showed that abelian $p$-groups do not.  They asked whether the class of torsion-free abelian groups lies on top. After quite some time, Shelah and Paolini \cite{PS} gave an affirmative answer. Independently and around the same time, Laskowski and Ulrich \cite{LU, LU2} provided an alternative proof by showing a more general result that certain classes of countable $R$-modules lie on top.

Kechris suggested that it would be good to develop an effective version of the notion of Borel embedding. The definition below is from \cite{CCKM}. That paper includes results on classes of finite structures, and the structures are allowed to have universe a proper subset of $\omega$.  As we said earlier, our structures will have universe~$\omega$. 

\begin{defn}
\label{defn:TCembedding}
Suppose $K\subseteq Mod(L), K'\subseteq Mod(L')$ are closed under isomorphism, and $L,L'$ are computable languages.  Suppose that the classes $K,K'$ are both closed under isomorphism. A \emph{Turing-computable embedding} (or \emph{$tc$-embedding}) of $K$ in $K'$ is a Turing operator $\Phi:K\rightarrow K'$ such that $\mathcal{A}\cong\mathcal{B}$ if and only if $\Phi(\mathcal{A})\cong\Phi(\mathcal{B})$.   If there is such an embedding, we write $K\leq_{tc} K'$.  The embedding is called a \emph{$tc$-reduction} from $K$ to $K'$.
\end{defn}

The proof that an operator is one-to-one on isomorphism types often involves showing that the input structure is interpreted (in a uniform way) in the output structure.  
In \cite{M14}, Montalb\'{a}n defined a very general notion of effective interpretation,
in which the interpreting formulas have no fixed arity.  A \emph{generalized} computable $\Sigma_1$
formula is a c.e.\ disjunction of existential formulas, possibly of different arities.
For example, there is a generalized computable $\Sigma_1$ formula that defines dependence
in $\mathbb{Q}$-vector spaces of tuples of all finite lengths.

\begin{defn}[Definition 5.1 in \cite{M14}]
\label{defn:effectiveinterpretation}
Let $\mathcal{A}$ be an $L$-structure, and let $\mathcal{B}$ be an $L'$-structure.  For simplicity, we suppose that $L$ is a finite relational language.  An \emph{effective interpretation of $\mathcal{A}$ in $\mathcal{B}$} is a tuple of generalized computable $\Sigma_1$ formulas defining a set $D\subseteq \mathcal{B}^{<\omega}$, an equivalence relation $\sim$ on $D$, the complementary relation $\not\sim$, and, for each $n$-ary relation symbol $R$ of $L$, an $n$-ary relation $R^*$ on $D$, and the complementary relation $\neg(R^*)$, such that the quotient structure $(D,R^*_{R\in L})/_{\!\sim}$ is well-defined and isomorphic to $\mathcal{A}$.  

\end{defn}  

We note that each computable structure $\mathcal{A}$ (with universe $\omega$) can be effectively interpreted in every infinite structure $\mathcal{B}$.  In the interpretation, the set $D$ is equal to $\mathcal{B}^{<\omega}$, and the element $n$ in $\mathcal{A}$ is represented by all $n$-tuples in $\mathcal{B}$.  An effective interpretation of an $L$-structure $\A$ in an $L'$-structure $\B$, as in Definition \ref{defn:effectiveinterpretation} above, gives a uniform effective method
of producing a copy of $\A$ from any copy of $\B$.  In some cases, the same formulas
define interpretations in many different structures $\B$.  For example, if $K$ is the class
of all countable fields (with domain $\omega$), there is an effective interpretation, uniform
for every $F\in K$, of the polynomial ring $F[X]$ in the field $F$.  (Notice that this cannot
be done using the ordinary model-theoretic notion of interpretation, with finitary formulas.)
In this way, a uniform effective interpretation may at times produce a Turing-computable
embedding of one class $K$ into another class $K'$, as in Definition \ref{defn:TCembedding}.
In the example above, the rings $E[X]$ and $F[X]$ are isomorphic if and only if the fields
$E$ and $F$ were isomorphic, and so this is indeed a $tc$-embedding.

Not all Turing-computable embeddings arise from effective 
interpretations. 
One example is the Friedman-Stanley embedding of graphs in linear orderings.  
However, in \cite{HTM^3}, Harrison-Trainor, Melnikov, Miller, and Montalb\'{a}n connected
effective interpretations with a notion recently formulated by Miller, Poonen, Schoutens,
and Shlapentokh in \cite{MPSS18}, which we now describe.

\begin{defn}
\label{defn:IsoK}

For a class $K$ of $L$-structures, we define the category $\Iso{K}$ to have as its objects
all structures in $K$, and to have as its morphisms all isomorphisms between objects in $K$.

Normally $K$ is closed under isomorphism (the structures all have universe $\omega$).
For a single structure $\A$, we define $\Iso{\A}$ to have $\set{\A'\cong\A}{\dom{\A'}=\omega}$
as its objects and all isomorphisms among these objects as the morphisms.  
\end{defn}

\begin{defn}[{\cite[Definition 3.1]{MPSS18}}]
\label{defn:functor}

A \emph{computable functor} from $\Iso{K'}$ to $\Iso{K}$ consists of two Turing functionals
$\Phi$ and $\Psi$ such that:
\begin{itemize}
\item
for every $\B\in\Iso{K'}$, $\Phi(\B)$ is (the atomic diagram of) a structure in $\Iso{K}$, 
\item
for every isomorphism $f:\B\to\B'$ in $\Iso{K'}$, $\Psi^{\B\oplus f\oplus \B'}$ is an isomorphism
from $\Phi(\B)$ onto $\Phi(\B')$, and 
\item
these two maps define a functor from $\Iso{K'}$ into $\Iso{K}$.  (Specifically, the map on morphisms
respects composition $\circ$ and preserves the identity isomorphism.)
\end{itemize}
A Turing-computable embedding $\Phi$ from $K'$ into $K$ is \emph{functorial} if there exists another
functional $\Psi$ such that $(\Phi,\Psi)$ is a computable functor from $\Iso{K'}$ into $\Iso{K}$.
\end{defn}

In a functor, $\Psi$ ensures that the map $\Phi$ on objects preserves the relation of being isomorphic.
In an embedding (whether Turing-computable or Borel), one requires $\Phi$ also to preserve non-isomorphism, so that
$\B\cong\B'$ if and only if $\Phi(\B)\cong\Phi(\B')$. As an example, the
$tc$-embedding described earlier, taking fields $F$ to rings $F[X]$, extends to a computable
functor in an obvious way.  Harrison-Trainor, Melnikov, Miller, and Montalb\'{a}n
showed that this reflects the uniform effective interpretability of the polynomial rings
in the fields.  The statement we give here is an amalgam of their Theorems 1.5 and 1.12
from \cite{HTM^3}.
\begin{thm}[\cite{HTM^3}]
\label{thm:equivalence}
Let $\Iso{\A}$ and $\Iso{\B}$ be categories as above.  Then effective interpretations
of $\A$ in $\B$ correspond bijectively to computable
functors from $\Iso{\B}$ into $\Iso{\A}$.  Each interpretation produces its functor
in the natural way, mapping each $\widetilde{\B}\cong\B$ to the copy of $\A$
given by the interpretation within $\widetilde{\B}$.

More generally, for categories $\Iso{K}$ and $\Iso{K'}$ as above, computable functors
from $\Iso{K}$ into $\Iso{K'}$ correspond bijectively to uniform effective interpretations
of $K'$ in $K$ in the same natural way.
(By a \emph{uniform effective interpretation}, we mean here a set of formulas
such that, for every $\B\in K$, the formulas give an effective interpretation in $\B$
of some $\A\in K'$.)
\end{thm}

In \cite{HTM^2}, these results are extended to cover functors and interpretations more broadly.
However, for our present investigations, Theorem \ref{thm:equivalence} will suffice.
Indeed, we will produce some new examples of computable functors, which can then
be converted into effective interpretations using the proof of Theorem \ref{thm:equivalence}
in \cite{HTM^3}.  In Section \ref{sec:countable}, we will consider a different sort of reduction between
isomorphism relations on classes of structures, in which only countably many 
structures from each class are considered at a time.  This will give results markedly different
from those of Hjorth and Thomas.  However, this type of reduction will not be used
before Section \ref{sec:countable}, so we postpone its description until that section.

\section{Scott sentences}
\label{sec:Scottsentences}

As mentioned in the introduction, we may describe a structure $\mathcal{A}$ in $\TFAb$ by giving a basis $\bar{a}$ and saying which $\mathbb{Q}$-linear combinations of $\bar{a}$ are present.  Similarly, we may describe a structure in $\FD$ by giving a basis $\bar{a}$ and saying which polynomials over $\mathbb{Q}(\bar{a})$ have roots. We will see that every structure $\mathcal{A}$ in $\TFAb_r$ or $\FD_r$ has a $\Sigma_3$ Scott sentence.  
We begin with torsion-free abelian groups.  We use the standard group language, with one binary operation symbol (for the group operation), one unary operation symbol (for inverse), and one constant symbol (for the identity).  Since the identity and inverses are definable without quantifiers, formulas in our language are equivalent to formulas of the same complexity in the language with just the operation symbol. 

\begin{prop}\label{group-Scott-sent}

Every torsion-free abelian group $G$ of finite rank has a $\Sigma_3$ Scott sentence, 
which is $G$-computable.

\end{prop}

\begin{proof}

Let $G$ be a subgroup of $\mathbb{Q}^n$ of full rank.  Consider an $n$-tuple $\bar{a}$ in $G$ that is linearly independent over $\mathbb{Z}$, and let $S$ be the set of $\mathbb{Q}$-linear expressions $\lambda(\bar{x})$ such that $\lambda(\bar{a})$ is present in $G$.  The Scott sentence says that $G$ is a torsion-free abelian group ($\Pi_1$), and that
there is a tuple $\bar{x}$ that is independent and such that the $\mathbb{Q}$-linear combinations present are just those in $S$.  To say that $\bar{x}$ is independent, we take the conjunction of formulas $m_1x_1 + \ldots + m_nx_n\not= 0$, for integers $m_1,\ldots,m_n$ not all $0$.  This is computable $\Pi_1$. To say that $y = \lambda(\bar{x})$ for $\lambda(\bar{x}) = q_1x_1 + \ldots + q_nx_n$, where $q_i = \frac{m_i}{m}$, we write $my = m_1x_1 + \ldots + m_nx_n$.  This is quantifier-free.  To say that the $\mathbb{Q}$-linear combinations present are just those in $S$, we write $\bigwedge_{\lambda\in S}(\exists y) y = \lambda(\bar{x})\ \&\ (\forall y)\bigvee_{\lambda\in S}y = \lambda(\bar{x})$.  This is $\Pi_2$.  With a $G$-oracle we can enumerate $S$ and compute the $\Pi_2$ sentence,
making the whole Scott sentence $G$-computable $\Sigma_3$.
\end{proof}

\begin{prop}\label{group-no-pi3}

There exists $G\in \TFAb_1$ with no $\Pi_3$ Scott sentence.  

\end{prop} 

\begin{proof}

By a result of Montalb\'{a}n \cite{Robuster}, there is a $\Pi_3$ Scott sentence if and only if the orbits of all tuples are defined by $\Sigma_2$ formulas.  We describe $G$ with a basis whose orbit is not defined by a $\Sigma_2$ formula.  Take $G\subseteq\mathbb{Q}$ generated by all $\frac{1}{p}$ (for all primes $p$).  For a basis, take $a = 1$.  This is divisible, just once, by each prime.  Suppose $(\exists\bar{u})\psi(x,\bar{u})$ defines the orbit of $1$ in $G$, where $\psi(x,\bar{u})$ is $\Pi_1$, and take $\bar{c}$ such that $G\models\psi(1,\bar{c})$.  Each $c_i$ has form $\frac{m_i}{n_i}$, where $m_i,n_i$ are relatively prime and $n_i$ is a product of primes, each occurring at most once.  
Take a prime $q$ not a factor of any $m_i$ or $n_i$.  Let $a' = \frac{1}{q}$ and let $c_i' = \frac{m_i}{qn_i}$.  We have an isomorphism $x\rightarrow\frac{x}{q}$ from $G$ onto the extension $G'$ generated by $\frac{1}{q^2}$ and the elements of $G$.  All formulas true in $G$ of $a,\bar{c}$ are true in $G'$ of $a',\bar{c}'$.  Since $G\subseteq G'$, the $\Pi_1$ formulas are true in $G$ of $a',\bar{c}'$.  So, we have $G\models(\exists\bar{u})\psi(a',\bar{u})$.  However, $a = 1$ is divisible by $q$ in $G$, while $a' = \frac{1}{q}$ is not.  So $a'$ is not in the orbit of $a$.  Thus, this orbit does not have a $\Sigma_2$ definition. 
\end{proof} 


In \cite{KLM}, it is shown that $\mathbb{Q}$ is the only rank $1$ torsion-free abelian group, up to isomorphism, with a $\Pi_2$ Scott sentence.

\bigskip

For fields, we use the usual field language with three binary operation symbols (for addition, subtraction, and multiplication) and two constant symbols (for $0$ and~$1$). We could omit subtraction and the constants $0$ and $1$ since these are defined by quantifier-free formulas using just addition and multiplication.

\begin{prop}
\label{field-Scott-sent}

Every field $F$ of characteristic 0 and finite transcendence degree has a $\Sigma_3$ Scott sentence,
which is $F$-computable.
\end{prop}

The proof of Proposition \ref{field-Scott-sent} is the same as that for Proposition \ref{group-Scott-sent}, with linear independence replaced by algebraic independence, and taking $S$ to be the set of existential formulas satisfied by the transcendence basis $\bar a$. 

\bigskip

We also adapt the proof of Proposition \ref{group-no-pi3} to the setting of fields. We use a coding that will appear again in the next section.

\begin{prop}

There exists $F \in \FD_1$ with no $\Pi_3$ Scott sentence. 

\end{prop}

\begin{proof}

As in the proof of Proposition \ref{group-no-pi3}, it suffices to produce $F\in \FD_1$ with a basis $a$ whose orbit is not defined by a $\Sigma_2$ formula.  Let $F \subseteq \overline{\Q(t)}$ be the subfield generated by the set of elements $t^{1/p}$, for $p>2$ prime, and let $a = t$.
(For definiteness:  within $\overline{\Q(t)}$, fix a real closed subfield containing $t$, and
choose each $t^{1/p}$ to be the unique $p$-th root of $t$ within that subfield.)
Suppose $(\exists\bar{u})\psi(x,\bar{u})$ defines the orbit of $t$ in $F$, where $\psi(x,\bar{u})$ is $\Pi_1$, and take $\bar{c}$ such that $F\models\psi(t,\bar{c})$.  Each $c_i$ is given by an algebraic expression in some finite collection of the generators $t^{1/p_{i,j}}$. 
Take a prime $q$ different from all of these $p_{i,j}$, and let $a' = t^{1/q}$. Let $F'$ be the extension of $F$ by $t^{1/q^2}$.  Then there is an isomorphism $f:F \to F'$ that fixes $\Q\subset F$ and sends $t$ to $t^{1/q}$. Let $c_i' = f(c_i)$. Since $f$ is an isomorphism and $F\models\psi(t,\bar{c})$, we have $F'\models\psi(t^{1/q},\bar{c'})$. On the other hand, since $F\subseteq F'$, any $\Pi_1$ formulas that are true in $F'$ of $a',\bar{c}'$ are also true in $F$ of $a',\bar{c}'$.  So, we have $F\models(\exists\bar{u})\psi(a',\bar{u})$.  However, $a = t$ has a $q$-th root in $F$, while $a' = t^{1/q}$ does not.  So $a'$ is not in the orbit of $a$.  Thus, this orbit does not have a $\Sigma_2$ definition. 
\end{proof}

\section{A functorial $tc$-embedding of $\tfa_r$ into $\TD_r$}
\label{sec:groupstofields}

As above, let $\TFAb_r$ be the class of torsion-free abelian groups of rank $r$. We view it
as a topological space of structures in the class with domain $\omega$, in the signature with $+$.
The topology arises from the identification of atomic diagrams with elements of $2^\omega$,
as described in Section \ref{sec:reducibility}.
Likewise, $\TD_r$ is the class of fields of characteristic $0$ and of transcendence degree $r$
over the prime subfield $\Q$.  As with $\TFAb_r$, we view it as a topological space
of structures in the class with domain $\omega$, in the signature with $+$ and $\cdot$,
as usual for fields.  The elements $0$ and $1$ and the operations of subtraction and division
are all definable by quantifier-free formulas, so may be used without hesitation.

$\TFAb_0$ is trivial.  In contrast, $\TD_0$ is not trivial:  there are many algebraic field
extensions of $\Q$, and they have been carefully studied.  For rank $r\geq 1$, however,
$\TFAb_r$ and $\TD_r$ show distinct similarities:  in both cases the isomorphism relation is effective $\Sigma_3$, and becomes effective $\Pi_2$ if one adds $r$ constant symbols to represent the elements of an arbitrary maximal independent set.  (\emph{Independence} refers to linear independence in $\TFAb_r$ and to algebraic independence in $\TD_r$.)  We also remark that the Turing degree spectra of groups in $\TFAb_r$ (for any fixed $r>0$)
are exactly the same as those of fields in $\TD_d$ (for any fixed $d\geq 0$):  in both cases,
they are exactly those sets $\set{\bfd}{S\in\Sigma_1^{\bfd}}$ of Turing degrees
defined by the ability to enumerate a specific set $S\subseteq\omega$.

\begin{thm}
\label{thm:TFAbtoTD}

For each finite rank $r>0$, there is a functorial Turing-computable embedding $(\Phi,\Phi_*)$
of $\TFAb_r$ into $\TD_r$, uniformly in $r$.

\end{thm}

The image of this functor is not all of $\TD_r$, and we do not claim that it has
a computable inverse functor.  So, this theorem explains this similarity between the spaces to some extent, but not completely. The rest of this section is dedicated to the proof of Theorem \ref{thm:TFAbtoTD}.

\begin{pf}

The input to the operator $\Phi$ consists of a rank $r>0$ and the atomic diagram (denoted $G$)
of a group from $\TFAb_r$.  Now, $\Phi$ is required to compute the atomic diagram of a field in $\TD_r$.  To avoid confusion, we write $\oplus$ for addition in $G$, and $+$ and $\cdot$ for the field operations in the output.  
Given a group $G$ with universe $\omega$, $\Phi$ 
names a corresponding collection of field elements $\la Y_n\ra_{n\in\omega}$ that we will call \emph{monomials}.  The multiplicative structure of the field on these elements is exactly that given by the group:  $Y_i\cdot Y_j = Y_k$ if and only if $i\oplus j=k$ in $G$.  We therefore view the indices
$i$, $j$, and $k$ as elements of $G$. (If $e\in G$ is the group identity element,
then $Y_e$ will be the element $1$; i.e., the multiplicative identity of the field.)

The elements of the field $\Phi(G)$ represent 
quotients
$$\frac{\sum a_n Y_n}{\sum b_nY_n}$$
of finite $\Q$-linear combinations of these monomials (including $Y_e=1$)
for which some $b_n\neq 0$.  Addition of two $\Q$-linear combinations
is done by treating the monomials as indeterminates.  Multiplication
uses the structure on the monomials:
\begin{equation}
\label{eq:mult}
\left(\sum a_m Y_m\right)\cdot\left(\sum b_nY_n\right) = \sum_k \left(\sum_{m\oplus n =k} a_mb_n\right)Y_k.
\end{equation}

This makes the $\Q$-linear combinations into a ring $R$, and
Lemma \ref{lemma:domain} below shows it to be an integral domain.

\begin{lemma}
\label{lemma:domain}

The ring $R$ described above is an integral domain.  

\end{lemma}

\begin{pf}

Consider the product in Equation \ref{eq:mult}.  Assume that all coefficients $a_m$ and $b_n$ are nonzero, and that the factors on the left are both 
nonzero (so neither sum is empty).  Fixing a linearly independent
set $U=\{ u_1,\ldots,u_r\}$ in $G$, we can express each index $m$ and $n$ from $G$
as a $\Q$-linear combination of $U$:  say $m=\sum p_{mi}u_i$ and $n=\sum q_{ni}u_i$.
Thus each $m$ corresponds to $(p_{m1},\ldots,p_{mr})\in\Q^r$,
and each $n$ to $(q_{n1},\ldots,q_{nr})\in\Q^r$.
Ordering these $r$-tuples lexicographically 
(and comparing individual coefficients under the usual order $<$ on $\Q$),
fix the particular $m_0$ for which $(p_{m_01},\ldots,p_{m_0r})$ is the
maximum of the set $\set{(p_{m1},\ldots,p_{mr})}{a_m\neq 0}$
and the $n_0$ for which $(q_{n_01},\ldots,q_{n_0r})$ is the
maximum of the set $\set{(q_{n1},\ldots,q_{nr})}{b_n\neq 0}$.
Let $k_0=m_0\oplus n_0$.  Then the coefficient of $Y_{k_0}$
in the product (on the right side of \ref{eq:mult}) is simply $a_{m_0}b_{n_0}$:
no other pair $(m,n)$ of these indices can have $m\oplus n=k_0$,
because $\oplus$ respects the lexicographic order we have chosen.
Since $a_{m_0}b_{n_0}\neq 0$, the product on the right is nonzero.  This proves the lemma. 
\qed\end{pf}

We define addition and multiplication on the formal quotients of ring elements $\frac{A}{B}$ (for $B\not= 0$) in the obvious way. We also define the obvious congruence relation $\sim$, where $\frac{A}{B} \sim \frac{A'}{B'}$ iff $AB' = A'B$.  Everything is computable in $G$.     
In this way, we obtain the quotient field of the integral domain, whose elements are the $\sim$-equivalence classes of formal quotients $\frac{A}{B}$.  We build an isomorphic copy $F = \Phi(G)$ with universe $\omega$, and computable in $G$.

To see that the field $F$ has transcendence degree $r$, let $U=\{ u_1,\ldots,u_r\}$
once again be linearly independent in $G$.  Set $X_i=Y_{u_i}$ for each $i\leq r$.
Now each $v\in G$ is a $\Q$-linear combination of the elements of $U$,
say $dv=\bigoplus_{i} c_iu_i$ using integers $c_i$ and $d\neq 0$, so the corresponding
$Y_v^d = Y_{u_1}^{c_1}\cdots Y_{u_r}^{c_r}=\prod X_i^{c_i}$.  Expressing every $Y_v$ in this form,
we see that the field is generated by rational powers of $X_1,\ldots,X_r$,
and thus has transcendence degree $\leq r$.
(This also explains why we refer to the $Y_i$'s as \emph{monomials}:
they are actual monomials in the roots of the basis elements $X_i$.)

We also claim that $\{ X_1,\ldots,X_r\}$ is algebraically independent in $F$.
Suppose some polynomial relation holds on these elements.
We may re-express each term in the relation as a single $Y_i$, e.g.,
$cX_1^2X_2^5 = cY_{u_1}^2Y_{u_2}^5=cY_{2u_1}Y_{5u_2}=cY_{2u_1\oplus 5u_2}$,
so that the entire polynomial equation becomes linear (over $\Q$) in the $Y_i$'s:  
We have \[0 = \sum c_1Y_i ,\] 
which implies that every $c_i=0$.  However,
since $U$ is independent in $G$, distinct terms $X_1^{p_1}\cdots X_r^{p_r}$ yield
distinct monomials $Y_k$ under this process, with $k=\sum p_iu_i$.
(Otherwise, we would have a nontrivial $\Q$-linear relation on $U$.)
Therefore, there were no repeated terms to combine when the polynomial equation was re-expressed as the $\Q$-linear equation $0 = \sum c_iY_i$, and so the coefficients $c_i$ (which must all be zero) are the original coefficients from the polynomial equation on $X_1,\ldots,X_r$.
Thus, these elements are indeed algebraically independent.

It may now be helpful to view $X_1,\ldots,X_r$ as algebraically independent positive real numbers,
and to assume that all roots $X_i^{\frac1d}$ used here are positive and real as well.
Then the entire field $F$ can be considered as a subfield of the real numbers.
It is clear that this construction is functorial, and that the functor is computable.
Indeed, if $\Phi_*$ is given a group isomorphism $g:G_0\to G_1$, then
$\Phi_*^{G_0\oplus g\oplus G_1}$ simply maps each $Y_n\in F_0=\Phi(G_0)$
to $Y_{g(n)}\in F_1=\Phi(G_1)$, and then extends this map to all $\Q$-linear combinations
of the monomials $Y_n$ in $F_0$ and finally to their quotients.  Since the monomials generate
$F_0$, this is entirely effective, and it preserves the identity and composition.

Functors must preserve the isomorphism relation on structures, of course,
since they map isomorphisms to isomorphisms.  However,
the theorem also requires the map $\Phi$ to preserve non-isomorphism:
$$ G_0\cong G_1 \iff \Phi(G_0)\cong\Phi(G_1),$$
so that $\Phi$ will be a Turing-computable embedding of $\TFAb_r$
into $\TD_r$ as well as being functorial.

So, it remains to show that when $G\not\cong\Gtilde$, we must get non-isomorphic fields $F=\Phi(G)$ and $\Ftilde=\Phi(\Gtilde)$ as outputs. From here on, let $f:F \to \Ftilde$ be a field isomorphism.  We need to show $G\cong \Gtilde$.  We continue to use the transcendence basis
$X_1,\ldots,X_r$ of $F$ built from a basis $U$ for $G$, and the transcendence basis
$\Xtilde_1,\ldots,\Xtilde_r$ for $\Ftilde$ built in the analogous manner from some basis of $\Gtilde$.
The basis $U$ of $G$ naturally induces an embedding from $G$ to $\Q^r$, so we consider $G$ as a subgroup of $\Q^r$ from now on (and similarly for $\Gtilde$).

\begin{defn}
\label{subgroups}\

\begin{enumerate}

\item Let $m:G\to F$ be the function that sends an element in $G$ to its associated monomial in $F$; i.e., $m(a_1,\cdots, a_r) =X_1^{a_1}\cdots X_r^{a_r} $. 

\item  Let $H \subseteq G$ be the set of group elements such that the corresponding monomial in $F$ is sent by $f$ to a monomial; i.e., $H = \{g\in G \mid f(m(g))\mbox{ is a monomial of $\Ftilde$}\}$.

\end{enumerate}
Similarly define $\tilde m:\Gtilde \to \Ftilde$ 
and $\Htilde\subseteq \Gtilde$.

\end{defn}

Recall that a subgroup $H$ of an additive group $G$ is said to be \emph{pure} if for every $g\in G$ and $k\in \Z$, if $kg$ is in $H$, then $g$ is in $H$.

\begin{prop}
\label{subgroup-isom}

For the groups $H\subseteq G$ and $\Htilde\subseteq \Gtilde$ in Definition \ref{subgroups}, $H$ is a pure subgroup of $G$, and $\Htilde$ is a pure subgroup of $\Gtilde$. Furthermore, $f$ maps $m(H)$ onto $\tilde m(\Htilde)$.  The restriction of $f$ taking $m(H)$ onto $\tilde m(\Htilde)$, is a bijection that respects field multiplication.  Hence, $f$ maps $H$ isomorphically onto $\Htilde$.

\end{prop}

\begin{proof}
Suppose $g\in G$ and $kg\in H$ for some $k \in \N$. Then $m(g)^k = m(kg)$, so $f(m(g))^k = f(m(kg))$ is a monomial. However, if $f(m(g))$ were not a monomial, then $f(m(g))^k$ would also not be a monomial. Thus, $f(m(g))$ must be a monomial, and $g\in H$. Hence, $H$ is a pure subgroup of $G$. Similarly, $\Htilde$ is a pure subgroup of~$\Gtilde$. 
By definition, every element of $m(H) \subseteq m(G)$ is a monomial. Thus, for every \linebreak $h\in H$, $f(m(H))$ and $f^{-1}(f(m(h))) = m(h)$ are both monomials, so $f(m(h))$ is in $\tilde m(\Htilde)$.  Then $f$ maps $m(H)$ to $\tilde m(\Htilde)$ in a well-defined way.  Since $f$ is injective and respects multiplication, the restriction that takes $m(H)$ to $m(\Htilde)$ is also injective and respects multiplication. For $a \in \tilde m(\Htilde)$, $f^{-1}(a)$ is a monomial, and ${f(f^{-1}(a)) = a}$ is also a monomial. Thus, $f^{-1}(a) \in m(H)$.  This shows that $f$ maps $m(H)$ onto $m(\Htilde)$.
\end{proof}


Let $j$ be the rank of $H$ and $\tilde H$.  To show that $G\cong \tilde G$, we will show that $G\cong H\oplus Z^k$ and $\tilde G\cong \tilde H\oplus Z^k$, for $k = r - j$.  We will use the following notation.  For an element $g\in G$, let $|g|$ be the usual $r$-dimensional Euclidean norm of $g\in \Q^r$, and endow $G$ with the topology induced by the Euclidean distance. Let $\bar H$ be the $\Q$-span of $H$ in $\Q^r$. Since $H$ is a pure subgroup of $G$, we have that $\bar H\cap G = H$.  Let $\pi:\Q^r \to \Q^r/\bar H$ be the natural quotient map.  We will need a lemma by Schinzel.

\begin{lem}[{\cite[Lemma 1]{Sch}}]
\label{schinzel-lem}

Let $K$ be a field of characteristic 0. If $g\in K[x] \setminus \{0\}$ has in the algebraic closure of $K$ a nonzero root of multiplicity at least $m$, then the polynomial $g$ has at least $m+1$ terms (with nonzero coefficients). 

\end{lem}

\begin{lem}\label{small-coord}

There is some $\epsilon > 0$ such that for every nonzero $g\in G\setminus H$, $|g| > \epsilon$. 

\end{lem}

\begin{proof}

If $G = H$, then the statement is vacuously true. Thus, we assume that $G \neq H$. 
Toward a contradiction, assume that for every $\epsilon > 0$ there is some nonzero $g\in G\setminus H$ such that $|g| < \epsilon$. We claim that there is a sequence of elements $g_1, g_2, \cdots, g_k\in  G\setminus H$ such that:
\begin{enumerate}
\item $\pi(g_1), \cdots, \pi(g_{k-1})$ is linearly independent and $\pi(g_k)$ is a $\Q$-linear combination of them.
\item Let $\pi(g_k) = \sum_{j=1}^{k-1} q_j\pi(g_j)$. For every $1 \le i \le k$, let $f(m(g_i)) = \alpha_i/\beta_i$ and $T_i$ be the sum of the numbers of terms in $\alpha_i$ and $\beta_i$. Then $\sum_{j=1}^{k-1}|q_j|T_j < 1$.
\item For each $1\le i\le k$, $g_i$ is not divisible by any $n > 1$ in $G$. Namely, for each $n > 1$, there exists no $x$ with $n x = g_i$. 
\end{enumerate}

We construct the sequence by the following process:

\bigskip
\noindent
\textbf{Step 0}: Pick any $g_1\in G\setminus H$. 

\bigskip
\noindent
\textbf{Step 0.5}: We claim that there is a largest $n\ge 1$ so that $g_1$ is divisible by $n$, so that we can replace $g_1$ by $\frac 1n g_1$ to achieve (3). Since $g_1$ is not in $H$, we know $f(m(g_1))$ is not a monomial. Thus, we write $f(m(g_1)) = \alpha_1/\beta_1$ and let $T_1$ be the sum of the numbers of terms in $\alpha_1$ and $\beta_1$. Suppose there is some $g$ and $n > 1$ with $g_1 = ng$:  we claim that then $n < T_1$. Let $f(m(g)) = \alpha/\beta$, so $\alpha_1/\beta_1 = (\alpha/\beta)^n$. By replacing each $X_1, \cdots, X_r$ by $X^{t_1}, \cdots, X^{t_r}$ with $1 \ll t_1 \ll \cdots \ll t_r \in \N$, we may assume $\alpha(X^{t_1},\cdots, X^{t_r})$, $\alpha_i(X^{t_1},\cdots, X^{t_r})$, $\beta(X^{t_1},\cdots,X^{t_r})$, $\beta_i(X^{t_1},\cdots,X^{t_r})$ have the same number of terms as $\alpha, \alpha_1, \beta, \beta_1$, respectively, and $\alpha(X^{t_1},\cdots, X^{t_r})$ does not divide $\beta(X^{t_1},\cdots,X^{t_r})$. Since $g \notin H$, without loss of generality, assume $\alpha$ is not a monomial and choose a root $\xi$ of $\alpha(X^{t_1},\cdots, X^{t_r})$ that is not a root of $\beta(X^{t_1},\cdots, X^{t_r})$. Then $\xi$ is a root of $\alpha_1(X^{t_1},\cdots, X^{t_r})$ of multiplicity at least $n$. Thus, by Lemma \ref{schinzel-lem}, $\alpha_1$ has at least $n+1$ terms, which establishes $n+1 < T_1$. Hence there must be a largest $n$ that divides $g_1$.

\bigskip
\noindent
\textbf{Step 1}: If the current tuple $\pi(g_1), \pi(g_2), \dots, \pi(g_{i-1})$ is linearly independent, we proceed to find $g_i$. Define $\alpha_j, \beta_j, T_j$ as in (2) and $Q_j = 1/(iT_j)$. Let $S$ be the pure subgroup of $G/H$ generated by $\pi(g_1),\cdots, \pi(g_{i-1})$ and let $S_0 = \{\sum q_i\pi(g_i) : (\forall i) |q_i| < Q_i\}$. Then $S_0$ is an open set in $S$, so its preimage $\pi^{-1}(S_0)$ is also an open set in $\pi^{-1}(S)$. Thus, there is some $\epsilon_i$ such that the ball $B$ of radius $\epsilon_i$ around the origin satisfies $B \cap \pi^{-1}(S)\subset \pi^{-1}(S_0)$. By assumption, there is some $g_i \in G \setminus H$ with $|g_i| < \epsilon_i$. Use this as our next $g_i$. 

\bigskip
\noindent
\textbf{Step 1.5}: If $g_i$ is divisible by some $n > 1$, repeat the argument in Step 0.5: let $n >1$ be the largest number dividing $g_i$, and replace $g_i$ by $\frac 1n g_i$. 

\bigskip
\noindent
\textbf{Step 2}: If the current tuple $\pi(g_1), \pi(g_2), \dots, \pi(g_i)$ is linearly independent, return to Step 1 and find the next $g_{i+1}$. If $\pi(g_1), \pi(g_2), \dots, \pi(g_i)$ is linearly dependent, return the tuple. As the range of $\pi$ is a quotient of $\Q^r$, it is finite-dimensional.  Thus $\pi(g_1), \pi(g_2), \dots, \pi(g_i)$ must eventually be linearly dependent, for some $i$, so the process will halt. 

\bigskip

Note that the Step 1 and 2 loop guarantees (1), and Step 0.5 and 1.5 guarantee (3). Thus, we only need to check that the constructed sequence satisfies (2). As $\pi(g_1), \cdots, \pi(g_{k-1})$ is linearly independent and $\pi(g_1), \cdots, \pi(g_{k})$ is linearly dependent, there is a unique way to write $\pi(g_k) = \sum_{j=1}^{k-1} q_j\pi(g_j)$. By the choice in Step 1 (and the fact that Step 1.5 will only decrease the norm of $g_k$ so will not affect the containment), we have $g_k \in B \cap \pi^{-1}(S)\subset \pi^{-1}(S_0)$, so $\pi(g_k) \in S_0$.  This means that $\pi(g_k)$ can be written as a linear combination of $\pi(g_j)$ where the $j$-th coefficient is less than $Q_j$. However, such a linear combination is unique, so we must have $|q_j| < Q_j$ for every $j$ (by the choice of $S_0$) and $\sum_{j=1}^{k-1}|q_j|T_j < \sum_{j=1}^{k-1}Q_jT_j < \sum_{j=1}^{k-1}1/k < 1$.

Now, we have a sequence of nonzero $g_1, g_2, \cdots, g_k\in  G\setminus H$ satisfying (1) to (3). By clearing denominators in the $q_i$, we can find some $n_i\in \Z$ such that $n_k \neq 0$ and $\sum n_i\pi(g_i) = 0$. By replacing $g_i$ with $-g_i$ if necessary, we will assume each $n_i \ge 0$ for simplicity. 
From (2), we have $\sum_{i=1}^{k-1}(n_i/n_k)T_j < 1$. 

Define $h = \sum n_ig_i$. We have $\pi(h) = 0$, so $h \in \bar H$, but also $h \in G$, so we have $h\in H$. Thus, $f(m(h))$ is a monomial. Now working in the field, we have
$$ m(h) = \vec X^h = \vec X^{n_ig_i} = \prod (\vec X^{g_i})^{n_i} = \prod(m(g_i))^{n_i}.$$
Taking $f$, we then have
$$ f(m(h)) = \prod (\alpha_i/\beta_i)^{n_i}.$$

Now, replacing each $X_i$ by some appropriate $X_i^{s_i}$, we may assume that $\alpha_i, \beta_i \in \Q[X_1,\ldots,X_r]$ while keeping the number of terms in $\alpha_i$ and $\beta_i$ the same as before. We can further replace each $X_i$ by $X^{t_i}$ some $1 \ll t_1 \ll \cdots \ll t_r \in \N$, so that $\alpha_i(X^{t_1},\cdots, X^{t_r})$ and $\beta_i(X^{t_1},\cdots,X^{t_r})$ have the same number of terms as $\alpha_i$ and $\beta_i$, respectively. For notational simplicity, we will from now on write $\alpha_i = \alpha_i(X^{t_1},\cdots, X^{t_r}) \in \Q[X]$, etc.

Since $\alpha_k/\beta_k \neq 1 $ and they are not both monomials, by taking the reciprocal if necessary, there must be some root $0 \neq \xi \in \C$ of $\alpha_k$ that is not a root of $\beta_k$. Since each $\beta_i$ has at most $T_i$ terms, by Lemma \ref{schinzel-lem}, if $\xi$ is a root of $\beta_i$, its multiplicity is at most $T_i$. Thus, as a root of $\prod (\alpha_i/\beta_i)^{n_i}$, $\xi$ has multiplicity at least $n_k-\sum_{i = 1}^{k-1} n_iT_i$. Note that $n_k-\sum_{i = 1}^{k-1} n_iT_i = n_k(1-\sum_{i = 1}^{k-1} (n_i/n_k)T_i) > 0$ by (2). Thus, $\xi$ is a root of $f(m(h)) = \prod (\alpha_i/\beta_i)^{q_i}$. However, $f(m(h)) = f(m(h))(X^{t_1},\cdots, X^{t_r})$ is a monomial and has no nonzero root, a contradiction. Thus, the lemma follows.
\end{proof}

\begin{lem}\label{small-coord-conseq}
Suppose there is some $\epsilon>0$ such that for every nonzero $g\in G\setminus H$, $|g| > \epsilon$. Then $G = H\oplus \Z^k$ for some $k\in \N$. 
\end{lem}

\begin{proof}
We first work in $\Q^r/\bar H$. Since there is a ball that intersects $G\setminus H$ trivially, the image $\pi(G)$ is discrete. Thus, $\pi(G)$ is isomorphic to $\Z^k$ for some $k$. Let $g_1, \cdots, g_k \in G$ such that $\pi(g_1), \cdots, \pi(g_k)$ is a free generating set of $\pi(G)$. Note that $g_i$ are linearly independent. 

We now show that $G = H \oplus \langle g_1\rangle \oplus\cdots\oplus\langle g_k\rangle$. Let $g \in G$.  Then $\pi(g) = \sum q_i\pi(g_i)$ for some $q_i\in \N$. Then $g-\sum q_ig_i$ is in $\bar H = \ker(\pi)$ and is also in $G$, so $g-\sum q_ig_i\in H$. Thus $g \in H+ \langle g_1\rangle +\cdots+\langle g_k\rangle $. Now, suppose for some $h,h' \in H$ and $q_i,q_i' \in \Z$, we have $h+\sum q_ig_i = h'+\sum q_i' g_i$.  Considering $\pi(G)$ and recalling that the elements 
$\pi(g_i)$ form a free generating set,
we must have $q_i = q_i'$. Thus, by canceling, we also have $h = h'$, so the sum $H+ \langle g_1\rangle +\cdots+\langle g_k\rangle$  is direct. Thus, $G = H \oplus \langle g_1\rangle \oplus\cdots\oplus\langle g_k\rangle \cong H \oplus \Z^k$.
\end{proof}

Combining Lemmas \ref{small-coord} and \ref{small-coord-conseq}, we see that $G= H\oplus \Z^k$. On the other hand, $\Gtilde$ satisfies exactly the same conditions with $f^{-1}:\Ftilde \to F$ a field isomorphism, so $\Gtilde = \Htilde\oplus \Z^{\tilde{k}}$. Since $H$ and $\Htilde$ are pure groups, we have $\rk(G) = \rk(H)+k$ and $\rk(\Gtilde) = \rk(\Htilde)+\tilde k$. We also have $H \cong \Htilde$ by Proposition \ref{subgroup-isom}, so $k = \tilde k$. Finally, we have that if $F\cong \Ftilde$, then $G= H\oplus \Z^k \cong \Htilde\oplus \Z^{\tilde k}= \Gtilde$. This completes the proof of Theorem \ref{thm:TFAbtoTD}.\qed
\end{pf}

In the proof, we need to fix an embedding of $G$ into $\Q^r$ (the paragraph before Definition \ref{subgroups}). However, this requires having a basis for $G$, which cannot in general be found computably. Furthermore, we also need to find a generating set of $\pi(G) \cong \Z^k$ in Lemma \ref{small-coord-conseq}, which also may not be computable.  The following question remains open.

\begin{question}
If $G, H\in \tfa_r$ and $g: \Phi(G) \to \Phi(H)$ is an isomorphism, then $G\cong H$ by Theorem \ref{thm:TFAbtoTD},
and by relative computable categoricity there must exist a $(G\oplus H)$-computable isomorphism $f:G \to H$.
Can we compute such an isomorphism uniformly from $G$, $H$, and $g$?
\end{question}

\section{$\TD_r$ into $\TD_{r+1}$}
\label{sec:fieldsup}

In this section, we show that for every $r \ge 0$, we have $\FD_r \le_{tc}\FD_{r+1}$ via a computable functor. For $\tfa$, if $A,B \in \tfa_r$, then $A \cong B$ if and only if $A\oplus \Z \cong B \oplus \Z$. Thus, $\Phi(A) = A\oplus \Z$ gives a Turing computable embedding. However, for $\FD$, there are two fields $E \not\cong F\in \FD_r$ such that $E(t) \cong F(t)$ (see \cite{Be85}). We will use the Henselization of a field to define a Turing computable embedding. 

We first consider the case when $r= 0$. In this case, the purely transcendental extension suffices. 

\begin{prop}
\label{prop:FD0case}

$\FD_0\leq_{tc} \FD_{1}$.  Furthermore, the Turing computable embedding is functorial (i.e., it can be extended to a computable functor).

\end{prop}

\begin{proof}

Consider a Turing operator $\Phi$ that takes $A\in \FD_0$ to a purely transcendental extension $A(t)$.  For this, we need to know when one rational function $f(t)$ (with coefficients in $A$) is equal to another $g(t)$.  It is enough to know when one polynomial $p(t)$ is equal to another $q(t)$.  This happens just when the difference is equal to $0$. In the purely transcendental extension, this happens exactly when each coefficient is 0. 

\begin{lem} 

For $A,A'\in \FD_0 $, if $f$ is an isomorphism from $A(t)$ onto $A'(t')$, then $A \cong A'$ via $f|_A$.    

\end{lem}

\begin{proof}

Let $f(t) = x$, and let $f(A) = B$.  Since $A(t)$ is a purely transcendental extension of $A$, $B(x)$ is a purely transcendental extension of $B$,
and the image $B(x)$ of $f$ must equal $A'(t')$.  Then $A'$ and $B$ are algebraic over $\mathbb{Q}$.  Now, Luroth's Theorem says that $A'$ is relatively algebraically closed in $A'(t)$.  That is, if $c\in A'(t)$ is algebraic over $\mathbb{Q}$, then it is already present in $A'$.  Thus, $B\subseteq A'$.  Similarly, $B$ is relatively algebraically closed in $B(x)$; if $c\in B(x)$ is algebraic over $\mathbb{Q}$, then it is already present in $B$.  Therefore, $A'\subseteq B$.  So, $f$ maps $A$ isomorphically onto $A'$. 
\end{proof}

Conversely, it is clear that if $A\cong A'$, then $A(t)\cong A'(t')$.  Moreover, whenever $f:A\to A'$ is an isomorphism,
we can extend $f$ to an isomorphism $\overline{f}:A(t)\to A'(t')$ by defining $f(t)=t'$.  Since $A(t)$ is constructed
so that its subset $A$ of constant rational functions is uniformly decidable within $A(t)$, this $\overline{f}$
is computable uniformly from $f$, $A$, and $A'$.  The choice of $\overline{f}$ respects composition
and preserves the identity, so we have a computable functor from $\FD_0$ to $\FD_1$.
\end{proof}

Similar to the results of Hjorth and Thomas on $\tfa_r$, this embedding is strict.

\begin{prop}

$\FD_1\not\leq_{tc} \FD_0$.

\end{prop} 

\begin{proof}

Existential sentences, saying which polynomials over $\mathbb{Q}$ have roots, are enough to distinguish non-isomorphic elements of $\FD_0$.  We will use the Pullback Theorem \cite{KMV}.  First, we show that there are non-isomorphic elements of $\FD_1$ with the same existential theory.  For this, consider a chain of three fields.  The first, $A_0$, is the algebraic closure of $\mathbb{Q}$.  The second, $A_1$, is $A_0(t)$, a purely transcendental extension of $A_0$.  The third, $A_2$, is the algebraic closure of $A_1$.  Now, $A_0$ and $A_2$ satisfy the same theory---that of algebraically closed fields of characteristic $0$.  Existential sentences are preserved under extension, so those true in $A_0$ are true in $A_1$ and those true in $A_1$ are true in $A_2$, matching those true in $A_0$.  Then $A_1$ and $A_2$ are non-isomorphic elements of $\FD_1$ with the same existential theory.  If $\Phi$ were a $tc$-reduction to $\FD_0$,
we would have $\Phi(A_1)\not\cong\Phi(A_2)$, so there would be an existential sentence $\varphi$ true in just one of the two.  The pullback $\varphi^*$ is a computable $\Sigma_1$ sentence true in just one of $A_1$ and $A_2$.  Now, $\varphi^*$ is a disjunction of existential sentences, one of which is true in just one of $A_1,A_2$.  This is a contradiction.  
\end{proof}

We can extend the previous result to Borel embeddings.

\begin{prop}

$\FD_1\not\leq_B \FD_0$.

\end{prop}

\begin{proof}  

Suppose $\Phi$ is a Borel embedding of $\FD_1$ in $\FD_0$.  

\bigskip 
\noindent
\textbf{Claim 1:} There is a Borel reduction of isomorphism on $\FD_0$ to $=$ (equality on sets).  (In fact, the output set is $\Delta^0_2$ uniformly relative to the input field.)

\begin{proof} [Proof of Claim 1]

Let $(p_n)_{n\in\omega}$ be a computable list of the polynomials $p(x)$ with coefficients in $\mathbb{Z}$.  Let $\Psi$ take the field $\mathcal{A}\in \FD_0$ to the set \mbox{$S = \{n:\mathcal{A}\models(\exists x)p_n(x)= 0\}$.}  We have $\mathcal{A}\cong\mathcal{A}'$ if and only if $\Psi(\mathcal{A}) = \Psi(\mathcal{A}')$. 
\end{proof}

Recall that $E_0$ (Vitali equivalence) is the equivalence relation on $2^\omega$ such that $fE_0 g$ iff $f$ and $g$ differ finitely; i.e., for all sufficiently large $n$, $f(n) = g(n)$.

\bigskip
\noindent
\textbf{Claim 2}:  There is a Borel reduction, even a $tc$-reduction, of $E_0$ 
to isomorphism on $\TFAb_1$.  
\begin{proof} [Proof of Claim 2]

Let $(p_n)_{n\in\omega}$ be a computable enumeration of the primes.  Let $\Psi'$ take $f$ 
to a computable copy of a subgroup of $\mathbb{Q}$ generated by the elements $\frac{1}{p_n}$ such that $f(n) = 1$.  
\end{proof}

We have shown that $\TFAb_1\leq_{tc} \FD_1$.  Composing the known reductions from $=^*$ to $\TFAb_1$, from $\TFAb_1$ to $\FD_1$, the purported reduction from $\FD_1$ to $\FD_0$, and the known reduction from $\FD_0$ to $=$, we would get a Borel reduction from $=^*$ to $=$. However, it is known that there is no such reduction.  
\end{proof}  

For general $r$, the map $F \mapsto F(t)$ no longer preserves (non-)isomorphism. Thus, we use the Henselization of a field to give a Turing computable embedding from $\FD_r$ to $\FD_{r+1}$. We first introduce some basic notions in valued fields and Henselization. We will use this as a black box and refer the reader to \cite{En05} for more detail. 
Given a field $K$ and a totally ordered abelian group $\Gamma$, we extend the group operation and ordering of $\Gamma$ naturally to $\Gamma \cup \{\infty\}$.  A \emph{valuation} of $K$ (with value group $\Gamma$), is a surjective map $v: K \to \Gamma \cup \{\infty\}$ such that for $a,b\in K$, (1) $v(a) = \infty$ if and only if $a = 0$, (2) $v(ab) = v(a)+v(b)$, and (3) $v(a+b) \ge \min (v(a),v(b))$.  Then $(K,v)$ is called a \emph{valued field}. We define the \emph{valuation ring} $O \subset K$ by $O = \{a\in K \mid v(a) \ge 0\}$. We say that $(K,v)$ is \emph{henselian} if $v$ has a unique extension to every algebraic extension $K'$ of $K$.

For convenience, we shall take the following characterization of the henselization of a valued field \cite[Theorem 5.2.2]{En05} as our definition.

\begin{defn}

Let $(K,v)$ be a valued field. Then the \emph{henselization} $(K^h,v^h)$ of $(K,v)$ is defined to be the valued field extension of $(K,v)$ such that
\begin{enumerate}
\item $(K^h,v^h)$ is henselian, and
\item for every henselian valued extension $(K',v')$ of $(K,v)$, there is a unique $K$-embedding $i: (K^h,v^h) \to (K',v')$. 
\end{enumerate}

\end{defn}

Every valued field has a henselization that is algebraic and is unique up to isomorphism. Furthermore, if $\Gamma$ and $\Gamma^h$ are the value groups of $(K,v)$ and the henselization $(K^h, v^h)$, we always have $\Gamma = \Gamma^h$ \cite[Theorem 5.2.5]{En05}.

\begin{defn}

A valued field is said to be \emph{discrete} if its value group is $\Z$

\end{defn}

We will need the following fact:

\begin{lem}[Folklore]\label{uniq-val}

A field can have at most one discrete Henselian valuation. 

\end{lem}

We can now describe the embedding of $\TD_r$ to $\TD_{r+1}$.  

\begin{thm}
\label{thm:fieldsupwards}
There is a functorial Turing computable embedding from $\TD_r$ to $\TD_{r+1}$.
\end{thm}

\begin{proof}

For 
$F\in \TD_r$, we let $\Phi(F)$ be the henselization of $F(t)$; i.e., \mbox{$\Phi(F) = F(t)^h$}. More precisely, we first apply a uniform effective procedure to pass from $F$ to the 
valued field $(F(t),v_t)$, where $v_t(p)$ is defined by $v_t(p) = \max\{n : t^n \mid p\}$ for a polynomial $p\in F[t]$, and $v_t(r) = v_t(p)-v_t(q)$ for $r = p/q\in F(t)$.  By \cite[Proposition 4]{Ha18}, there is a computable embedding of $(F(t),v_t)$ into a computable copy of $(F(t)^h,v_t^h)$. We take $\Phi(F)$ to be $F(t)^h$. 
It is clear that if $F \cong E$, then we have $\Phi(F) \cong \Phi(E)$. 

Now, suppose $\Phi(F) \cong \Phi(E)$.  By construction, $v_t$ is a discrete valuation of $F(t)$, so $v_t^h$ is also discrete \cite[Theorem 5.2.5]{En05}. Thus, by 
Lemma \ref{uniq-val}, we may let $v$ and $u$ be the unique Henselian valuations of $\Phi(E)$ and $\Phi(F)$, respectively.  Therefore, $(\Phi(F),v) \cong (\Phi(E),u)$ as valued fields. In $\Phi(F) = F(t)^h$, $F \cong v^{-1}[0,\infty]/v^{-1}(0,\infty]$.  Similarly, in $\Phi(E) = E(t)^h$, $E \cong u^{-1}[0,\infty]/u^{-1}(0,\infty]$. Thus, we have $F \cong E$. This shows that $\Phi$ is a Turing computable embedding.

To show that $\Phi$ is functorial, for some $f: F\cong E$, we take $\Phi_*(F\oplus f\oplus E)$ to be the map that takes $a\in F$ to $f(a) \in E$, and takes $t$ to $t$. Since $F(t)^h$ and $E(t)^h$ are algebraic over $F(t)$ and $E(t)$, respectively, we may construct an isomorphism $\Phi^*(F\oplus f\oplus E) = \tilde f$ from $F(t)^h$ to $E(t)^h$ by mapping roots of polynomials to corresponding roots. However, by \cite[Theorem 5.2.2]{En05}, there is a unique isomorphism between any two henselizations of a field, which must be the isomorphism $\tilde f$ we constructed. Thus, $\Phi^*$ is functorial and $(\Phi,\Phi_*)$ form a computable functor from $\TD_r$ into $\TD_{r+1}$.
\end{proof}

Now that we have $\FD_0 <_{tc} \FD_1 \le_{tc} \FD_2 \le_{tc}\dots$, it is natural to ask if we have strictness for $r \ge 1$, as in the case of torsion-free abelian groups. 
In the Borel setting, Thomas and Velickovic \cite{TV} showed that the class $\TD_{13}$ is universal among essentially countable Borel equivalence relations, so $\FD_r\le_B \FD_{13}$ for every $r$. They mentioned in their paper that they did not attempt to make the transcendence degree as low as possible, and they asked whether $\FD_1$ is already universal. However, the embedding from $\FD_r$ to $\FD_{13}$ induced by their proof is not computable, and it is still open whether there exists a Turing computable embedding.

\begin{question}
For which $r\in\omega$ is there a Turing computable embedding (possibly functorial?) from $\FD_{r+1}$ to $\FD_r$?
And for each $r=1,\ldots, 12$, is there a Borel embedding from $\FD_{r+1}$ to $\FD_r$?
\end{question} 

\section{From Fields to Groups}
\label{sec:fieldstogroups}

In Section \ref{sec:intro}, we saw that the isomorphism relation on $\FD_0$ is effective $\Pi_2$ and for $r > 0$, the isomorphism relation on $\TFAb_r$ and $\FD_r$ is effective $\Sigma_3$.  We promised to prove completeness, and we do that in Subsection \ref{sec:complete}.  In Subsection \ref{sec:non-embed}, we show that for $r > 0$, there is no functorial computable reduction from $\FD_r$ to $\TFAb_1$.   

\subsection{Completeness} \label{sec:complete}

When we say that a set (or relation) $A$ is ``complete'' for some complexity class $\Gamma$, we mean that $A$ is in $\Gamma$, and every set in $\Gamma$ is reducible to $A$ using a reduction function of the ``appropriate'' kind, so that the sets reducible to $A$ are \emph{exactly} those in $\Gamma$.  For sets of numbers and complexity classes $\Gamma$ in the arithmetical or hyperarithmetical hierarchy, the appropriate reduction functions are computable---\emph{complete} means \emph{$m$-complete}.  For $\Gamma$ consisting of sets $\Sigma_\alpha$ or $\Pi_\alpha$ relative to $X$, the appropriate reduction functions are $X$-computable.  
For subsets of $2^\omega$ and complexity classes $\Gamma$ in the effective Borel hierarchy, again the appropriate reduction functions are computable.  For $\Gamma$ in the Borel hierarchy, the appropriate reduction functions are continuous; i.e., $X$-computable for some $X$.  For more on this, see \cite{HKLM}.  We illustrate with a simple example.  

\begin{example}

Let $F$ be the field obtained by adding to $\mathbb{Q}$ a primitive root of each polynomial in a computable sequence $p_n(x)$, where the field generated by roots of $p_k(x)$ for $k < n$ does not have a root of $p_n(x)$.   We could could take $p_n(x)$ to be the cyclotomic polynomial $1 + x + \ldots + x^{k-1}$ where $k$ is the $n^{th}$ prime.  

\end{example}

Clearly, $F$ has a computable copy.  The set $I(F)$, consisting of indices for computable copies of $F$, is $\Pi^0_2$ in the arithmetical hierarchy.  It is complete $\Pi^0_2$.  To show this, it is enough to show that for the set $Inf = \{n:\mbox{$W_n$ is infinite}\}$, which is known to be complete $\Pi^0_2$, $Inf\leq_m I(F)$.  We define a uniformly computable sequence $(F_n)_{n\in\omega}$, where at stage $s$, we check the size of $W_{n,s}$ for each $n < s$.  If the size is $k$, then we put into $F_n$ primitive roots for the first $k$ polynomials, no more.  We know the indices for $F_n$, and our $m$-reduction takes $n$ to the index for $F_n$ such that $n\in Inf$ iff $F_n\cong F$.   

\begin{prop}

For our field $F$, $\Iso{F}$, the set of isomorphic copies of $F$, is complete effective $\Pi_2$.  
It is also complete $X$-effective $\Pi_2$ and complete $\mathbf{\Pi_2}$.

\end{prop}

\begin{proof}

Since $F\in \FD_0$, $\Iso{F}$ is effective $\Pi_2$.  We show that it is complete.  Let $D$ be an effective $\Pi_2$ set, with index $d$.  The index $d$ gives a c.e.\ set of indices for effective $\Sigma_1$ sets with intersection $D$.  Adding an index for $\omega$, if necessary, we pass effectively to a sequence of indices $d_n$ for effective $\Sigma_1$ sets $D_n$ such that $D = \cap_n D_n$.  We may assume that $D_0 = \omega$ and that the sets $D_n$ are nested.  We want a computable reduction $\Phi$ of $D$ to $\Iso{F}$.  This $\Phi$, defined on all of $2^\omega$, takes each $f$ to a field in $\FD_0$ such that $f\in D$ iff $\Phi(f)\cong F$.  To compute $\Phi(f)$, we start enumerating the diagram of a field that looks like $\mathbb{Q}$, and we add a primitive root for $p_n(x)$ if and when we see that $f\in D_n$.  Note that for all $f$, $\Phi(f)$ is in $\FD_0$.

Since $\Iso{F}$ is effective $\Pi_2$, it is $X$-effective $\Pi_2$ for all $X$.  Relativizing what we did above, we can show that it is complete $X$-effective $\Pi_2$.  For each $X$-effective $\Pi_2$ set $D$, we have an $X$-computable reduction of $D$ to $\Iso{F}$.  Moreover, the range of the reduction consists of fields in $\FD_0$.  The fact that $\Iso{F}$ is effective $\Pi_2$ means that it is $\mathbf{\Pi_2}$ in the Borel hierarchy.  To show that it is complete $\mathbf{\Pi_2}$, we note that every $\mathbf{\Pi_2}$ set $D$ is $X$-effective $\Pi_2$ for some $X$, and our $X$-computable reduction of $D$ to $\Iso{F}$ is continuous.     
\end{proof}      

\begin{prop}

The isomorphism relation on $\FD_0$ is complete effective $\Pi_2$.

\end{prop}

\begin{proof}

In Section 1, we saw that the isomorphism relation on $\FD_0$ is effective $\Pi_2$.  For completeness, take an effective $\Pi_2$ set $D$.  For the specific field $F$ in our example, we have a computable reduction $\Phi$ of $D$ to $\Iso{F}$, where for all $f\in 2^\omega$, $\Phi(f)\in FD_0$.  We get a reduction $\Psi$ of $D$ to the isomorphism relation on $\FD_0$, where $\Psi(f)$ is the pair $(F,\Phi(f))$.  
\end{proof}  

The effective Borel hierarchy is closely tied to the hyperarithmetical hierarchy.  In \cite{HKLM}, it is observed that for any function $f\in 2^\omega$ and any computable ordinal $\alpha$, the set $T_{\Sigma_\alpha}(f)$ ($T_{\Pi_\alpha}$) of indices for effective $\Sigma_\alpha$ ($\Pi_\alpha$) sets that contain $f$ is $\Sigma^0_\alpha$ 
($\Pi^0_\alpha$), uniformly in $f$.  In \cite{HKLM}, this is used in some completeness proofs, and we shall use it here.                

We turn to $\TFAb_1$.  In \cite{KS}, it is shown that for any computable subgroup of $\mathbb{Q}$ in which there is a computable set of primes that divide $1$ just finitely, but at least once, the set of indices for computable copies is $m$-complete $\Sigma^0_3$.  One example of such a group is the subgroup of $\mathbb{Q}$ generated by $\frac{1}{p}$ for primes $p$.  It is easy to see that $I(G)$ is $\Sigma^0_3$.  The completeness proof in \cite{KS} involves showing that $Cof\leq_m I(G)$.  The computable reduction, defined on all $n\in\omega$, gives a uniformly computable sequence of groups $(G_n)_{n\in\omega}$ such that $n\in Cof$ iff $G_n\cong G$.  Each $G_n$ is isomorphic to a subgroup of $\mathbb{Q}$, in which each prime divides $1$ at most once, and in $G_n$, the $k^{th}$ prime divides $1$ iff $k\in W_n$.          

\begin{prop}

Let $G$ be as above.  Then the set $\Iso{G}$ is complete effective $\Sigma_3$.  It is also complete $X$-effective $\Sigma_3$ for all $X$, and complete $\mathbf{\Sigma_3}$.

\end{prop}  

\begin{proof}

Suppose $D\subseteq 2^\omega$ is effective $\Sigma_3$, with index $d$.  We need a computable reduction $\Phi$ of $D$ to $\Iso{G}$.  The set $T_{\Sigma_3}(f)$ is $\Sigma^0_3$ relative to $f$, with index computed in a uniform way from $f$.
Given $f\in 2^\omega$, we relativize the construction from \cite{KS}, in a uniform way.  Letting $Cof^f = \{n:\mbox{$W_n^f$ is co-finite}\}$, we get an $f$-computable sequence of groups $(G_{f,n})_{n\in\omega}$, all in $\TFAb_1$, such that $G_{f,n}\cong G$ iff $n\in Cof^f$.  The set $Cof^f$ is complete among $\Sigma^0_3(f)$ sets, with reduction functions computed uniformly from $f$.  In particular, knowing the index $d$ for $D$, we can pass effectively from an index for $T_{\Sigma_3}(f)$ to a number $d'$ such that $d\in T_{\Sigma_3}(f)$ iff $d'\in Cof^f$.  All together, we have $f\in D$ iff $d\in T_{\Sigma_3}(f)$ iff $d'\in Cof^f$ iff $G_{f,d'}\cong G$.  So, we take $\Phi(f)$ to be $G_{f,d'}$.  

The fact that $\Iso{G}$ is effective $\Sigma_3$ implies that it is $X$-effective $\Sigma_3$ for all $X$, and it is $\mathbf{\Sigma_3}$.  To show that $\Iso{G}$ is complete $X$-effective $\Sigma_3$, we need, for each $X$-effective $\Sigma_3$ set $D$, an $X$-computable reduction $\Phi$ of $D$ to $\Iso{G}$.  We obtain $\Phi$ by relativizing the construction above.  Moreover, $\Phi:2^\omega\rightarrow \TFAb_1$.           
\end{proof} 

\begin{prop}

For $r > 0$, each of the classes $\TFAb_r$, $\FD_r$ contains a structure $\mathcal{A}$ for which $\Iso{\mathcal{A}}$ is complete effective $\Sigma_3$, complete $X$-effective $\Sigma_3$ for all $X$, and complete $\mathbf{\Sigma_3}$.   

\end{prop} 

\begin{proof}

We consider the groups first.  Let $G$ be the group in the previous proposition.  This is the structure we want in $\TFAb_1$.  For $r > 1$, we have a $tc$-embedding $\Psi$ of $\TFAb_1$ in $\TFAb_r$.  Let $\mathcal{A}$ be $\Psi(G)$.  We know that $\Iso{\mathcal{A}}$ is effective $\Sigma_3$.  For completeness, let $D$ be an effective ($X$-effective) $\Sigma_3$ set.
We have a computable ($X$-computable) reduction $\Phi$ of $D$ to $\Iso{G}$, where $\Phi:2^\omega\rightarrow \TFAb_1$.  For $f\in 2^\omega$, we have $f\in D$ iff $\Phi(f)\cong G$ iff $\Psi(\Phi(f))\cong\mathcal{A}$.  Thus, $\Psi(\Phi(f))$ is a computable ($X$-computable) reduction of $D$ to $\Iso{\mathcal{A}}$.  Moreover, for all $f$, $\Psi(\Phi(f))$ is in $\TFAb_r$.  Thus, $\Iso{\mathcal{A}}$ is complete effective $\Sigma_3$ and complete $X$-effective $\Sigma_3$.  Every $\mathbf{\Sigma_3}$ set $D$ is $X$-effective $\Sigma_3$ for some $X$, and our $X$-computable reduction of $D$ to $\Iso{\mathcal{A}}$ is continuous.  Therefore, $\Iso{\mathcal{A}}$ is also complete $\mathbf{\Sigma_3}$.  

We turn to the fields.  Composing our $tc$-embeddings of $\TFAb_1$ in $\FD_1$ and $\FD_n$ in $\FD_{n+1}$, we arrive at a $tc$-embedding $\Phi$ of $\TFAb_1$ in $\FD_r$.  Let $\mathcal{A} = \Phi(G)$.  Proceeding exactly as above, we get the fact that $\Iso{\mathcal{A}}$ is complete effective $\Sigma_3$, complete $X$-effective $\Sigma_3$, and complete $\mathbf{\Sigma_3}$.  
\end{proof}  

We turn to the isomorphism relation on the classes.   

\begin{prop}

For each $r > 0$, the isomorphism relation on the classes $\TFAb_r$, $\FD_r$ is complete effective $\Sigma_3$.  (It is also complete $X$-effective $\Sigma_3$ and complete $\mathbf{\Sigma_3}$.)

\end{prop}

\begin{proof}

We sketch the proof for $\TFAb_1$.  The proof for the other classes is the same.  The class $\TFAb_1$ is effective $\Pi_2$, and the isomorphism relation on $\TFAb_1$ is effective $\Sigma_3$.  For completeness, we use the fact that there is a specific group $G\in \TFAb_1$ for which $\Iso{G}$ is complete effective ($X$-effective) $\Sigma_3$, and that for each effective ($X$-effective) $\Sigma_3$ set $D$, we have a computable ($X$-computable) reduction $\Phi$ of $D$ to $\Iso{G}$, such that 
$\Phi:2^\omega\rightarrow \TFAb_1$.  Let $\Psi(f) = (G,\Phi(f))$.  We have $f\in D$ iff $\Phi(f)\cong G$, so  $\Psi$ is a computable ($X$-computable) reduction of $D$ to the isomorphism relation on $\TFAb_1$.  This shows that the isomorphism relation is complete effective ($X$-effective) $\Sigma_3$.  Take a set $D$ that is $\mathbf{\Sigma_3}$.  This is $X$-effective $\Sigma_3$ for some $X$.  Our $X$-computable reduction of $D$ to $\Iso{G}$ is continuous.  
\end{proof} 


\subsection{Non-embeddability} \label{sec:non-embed}

This section is devoted to proving a first step in answering the question of whether
there exist $tc$-reductions in the opposite direction, from $\TD_r$ to $\TFAb_r$.
Our result here will exclude functorial $tc$-reductions
for the case $r=1$ (indeed from $\TD_r$ to $\TFAb_1$ for each $r\geq 1$)
but it uses a specific fact about $\TFAb_1$ that fails for $r>1$ and also fails
in every $\TD_r$:  an automorphism of a group in $\TFAb_1$
that fixes a single non-identity element must be the identity automorphism.
Therefore, we remain uncertain whether this theorem can be extended to $\TFAb_r$
for $r>1$, let alone whether it holds when $r>0$ and the condition of functoriality
is dropped.

\begin{thm}
\label{thm:r=1}

For each $r>0$, there is no functorial computable reduction from $\FD_r$ to $\TFAb_1$.

\end{thm}


\begin{proof} 

Suppose that $(\Phi,\Phi_*)$ were such a functorial reduction.  Fix a presentation    $A\in\TD_r$ of the purely transcendental extension $\Q(t_1,\ldots,t_r)$ of the rationals.
By functoriality $\Phi_*^{A\oplus\text{id}\oplus A}$ must be the identity map on 
$\Phi(A)$, so fix an initial segment $\sigma$ of (the atomic diagram of) $A$ sufficiently long that there exist three distinct elements $b_0,b_1,b_2\in\Phi(A)$ with
$\Phi_*^{\sigma\oplus (\text{id}\upharpoonright|\sigma|) \oplus\sigma}(b_i)=b_i$.
It is consistent for us to extend $\sigma$ to the atomic diagram of a copy of $\Q$.
(This could fail for certain other fields in $\TD_r$, but since $A\cong\Q(t_1,\ldots,t_r)$, it must hold.)
Let $q_0,\ldots,q_k\in\Q$ be all of the (finitely many) rational numbers mentioned in $\sigma$ when $\sigma$ is viewed as an initial segment of the diagram of this copy of $\Q$, and let $a_i$ (for each $i\leq k$) be the element of $\omega$ representing the rational $q_i$ in this copy.

Now, consider the following procedure for determining the relation of isomorphism
on fields in $\TD_r$.  To get a contradiction, given any two fields $E,F\in\TD_r$, we will reduce the question of whether $E \cong F$ to an effective $\Pi_2$ property.
We use the atomic diagrams of $E,F$ 
to find the rationals $q_0,\ldots,q_k$ in each.  We construct a permutation $f_0$ of $\omega$ that is the identity on all but finitely many elements, but (for each $i$) maps the domain element of $F$ representing $q_i$ to $a_i$.  Thus the field $F_0$, built so that $f_0:F\to F_0$ is an isomorphism, has $\sigma$ as an initial segment of its atomic diagram, with the domain element $a_i\in F_0$ representing the rational $q_i$ for each $i$.  We similarly construct an isomorphism $e_0$ mapping $E$ onto another field $E_0$ with initial segment~$\sigma$.

Now we consider the two groups $G=\Phi(E_0)$ and $H=\Phi(F_0)$ in $\TFAb_1$.
Suppose $E\cong F$.  Then there exists an isomorphism $f:E_0\to F_0$, and
$g=\Phi_*^{E_0\oplus f\oplus F_0}$ will be a group isomorphism from $G$ onto $H$.
However, $E_0$ and $F_0$ both have initial segment $\sigma$, and since each element $a_i$
mentioned in $\sigma$ represents the rational $q_i$ in both $E_0$ and $F_0$,
the isomorphism $f$ must have $f(a_i)=a_i$ for all these $i$.  Therefore
$$g(b_i)=\Phi_*^{\sigma\oplus(\text{id}\upharpoonright |\sigma|)\oplus \sigma}(b_i)=b_i$$
for each of the elements $b_0,b_1, b_2$ described earlier.  Since $g$ is an isomorphism, each of these 
elements $b_i$ individually satisfies
$$ (\forall \text{~prime powers~}p^m)~[((\exists x\in G)~p^mx=b_i) \iff (\exists y\in H)~p^my=g(b_i)],$$
hence, also satisfies  
\begin{equation}
\label{eq:iso}
(\forall \text{~prime powers~}p^m)~[((\exists x\in G)~p^mx=b_i) \iff (\exists y\in H)~p^my=b_i].
\end{equation}

Conversely, if $G,H\in\TFAb_1$ satisfy Equation \ref{eq:iso} for both of these elements $b_i$, then $G\cong H$.  (One of $b_0,b_1,b_2$ could be the identity element in $G$, and another could be the identity in $H$.  However, as the three are distinct, the remaining element suffices to establish isomorphism between these rank-$1$ groups.)  Since $\Phi$ is a Turing-computable embedding, $G\cong H$
implies $E\cong F$, completing the converse.  Thus $E\cong F$ if and only if
Equation \ref{eq:iso} holds for both $b_0$ and $b_1$.

Thus, we have effectively reduced the question of isomorphism between $E$ and $F$
to the $\Pi^0_2$ property given above as Equation \ref{eq:iso}, using only finitely much information:
the three elements $b_0,b_1,b_2$, the rationals $q_0,\ldots,q_k$, and the elements $a_0,\ldots,a_k$ of $\omega$.  
This is impossible, since the isomorphism relation on $\FD_r$ is complete effective~$\Sigma_3$. 
\end{proof}

\section{Countable reduction} 
\label{sec:countable}


In \cite{MSEALS} the third author introduced the following definition of (computable) $\mu$-ary reduction,
an extension of a notion originally proposed in \cite{MN16}, by himself and Ng,
in which $\mu$ was assumed to be finite.

\begin{defn}[{\cite[Definition 1.3]{MSEALS}}]
Let $E$ and $F$ be equivalence relations on $S$ and $T$. For any cardinal $\mu < |S|$, we say a function $g:S^\mu \to T^\mu$ is a \emph{$\mu$-ary reduction of $E$ to $F$} if for every $\xvec = (x_\alpha)_{\alpha\in\mu} \in S^\mu$, we have
$$ \forall \alpha < \beta < \mu \ (x_\alpha E x_\beta \Leftrightarrow g_{\alpha}(\xvec) F g_{\beta}(\xvec)) $$
where $g_\alpha$ is the $\alpha$-th component of $g$. 

When $S \subseteq 2^\omega$ and $T\subseteq 2^\omega$, $\mu\leq\omega$, and $g$ is computable, we write $E \le^\mu_0 F$ and call this a \emph{computable $\mu$-ary reduction.} 
\end{defn}

In this section, we focus on computable countable reducibility, i.e., the case $\mu = \omega$.
Notice that when a 
Turing computable reduction exists, we automatically have a computable countable reduction.  For example, the $tc$-reduction $\Phi$ given in Proposition \ref{prop:FD0case} yields a 
computable countable
reduction $g:(\FD_0)^\omega\to (\FD_1)^{\omega}$, where
$$ g(F_0,F_1,F_2,\ldots) = (\Phi(F_0),\Phi(F_1),\Phi(F_2),\ldots).$$
On the other hand, often a computable countable reduction exists even when there is no Turing computable  
reduction.  We will see that under computable countable reducibility, all of
the isomorphism relations on $\tfa_r$ and $\FD_r$ are equivalent for all $r \ge 1$,
whereas Theorem \ref{thm:HjorthThomas} shows that the same fails to hold under 
Turing computable reducibility.
In fact, Theorem \ref{tfab-sigma3} will show that, under computable countable
reducibility, the isomorphism relation on $\TFAb_1$ has
the maximal possible complexity for effective $\Sigma_3$ equivalence relations on $2^\omega$.
Intuitively this suggests that the non-reducibility results of Hjorth and Thomas
(summarized as Theorem \ref{thm:HjorthThomas})
do not stem from mere syntactic complexity, but instead depend intimately on
the uncountable nature of the spaces in question.  (If the universe were collapsed
by a forcing extension, so that the original sets $\TFAb_r$ became countable,
then the extension would contain a full computable reduction from each original
$\TFAb_{r+1}$ to the preceding $\TFAb_r$, given by the procedure in
Theorem \ref{tfab-sigma3} --- although of course, in the extension,
the original set $\TFAb_{r+1}$ would be superseded by a larger collection of
torsion-free abelian groups of rank $r+1$.)

In \cite{MSEALS}, it is shown that the equivalence relations $E_0, E_1, E_2$ are all $\Sigma^0_2$-complete under computable countable reduction.   

\begin{remark}
Since we are working in the computable setting, the representation of structures do change the complexity. In particular, if $G\in \TFAb_r$ are represented as a subset of $\Q^r$, then the divisibility predicate $n \mid g$ becomes computable (checking if $g/n\in G$) and the complexity of the isomorphism becomes $\Sigma^0_2$ (actually complete). However, in this paper, the structure are represented as free-standing structures, i.e., a point in $Mod(L)$, thus the divisibility predicate is $\Sigma^0_1$ and the complexity of isomorphism is $\Sigma^0_3$. 
\end{remark}

We now prove that the isomorphism relation on $\tfa_1$ is $\Sigma^0_3$-complete
under computable countable reduction. As a result, for every $r\ge 1$,
there is a computable countable reduction from $\FD_r$ and $\tfa_r$ to $\tfa_1$. 

\begin{thm}\label{tfab-sigma3}
$\tfa_1$ is $\Sigma^0_3$-complete under computable countable reducibility. More precisely, every $\Sigma^0_3$ equivalence relation $E$ on a subspace of $2^\omega$ is computably countably reducible to $\tfa_1$.
\end{thm}

\begin{pf}
We first observe that $\tfa_1$ is defined by the following computable infinitary
$\Sigma^0_3$ formula on a subset of $2^\omega$: For $G,H\in \tfa_1$, $G\cong H$ if and only if
$$ \exists g\in G, h\in H, \forall q\in \Q [(\exists g'\in G~g' = qg) \Leftrightarrow (\exists h'\in H~h' = qh)].$$

Let $E$ be a $\Sigma^0_3$ equivalence relation on a subset of $2^\omega$. We may assume that
 $AEB$ if and only if $\exists x \forall y \exists z R(A,B,x,y,z)$ where $R(A,B,x,y,z)$ is a computable predicate.
 
Recall that our reduction requires a Turing functional $\Phi$ that accepts as an oracle
the join $A_0\oplus A_1\oplus\cdots$ of countably many sets in $2^\omega$ and 
(assuming every $A_n$ is in the field of the equivalence relation $E$)
outputs the join $G_0\oplus G_1\oplus\cdots$ of the atomic
diagrams of countably many groups in $\TFAb_1$, so that
$$ (\forall n<m)~[A_m E A_n \iff G_m\cong G_n].$$

In the construction, we will consider each $G_i$ as a subgroup of $\Q$. In fact, $\Phi$ will first build subgroups $G_i$ of $\Q$, and then turn each of them into its atomic diagram. 

Notice that for a given $m,n,k\in\omega$, the property $\exists x\le k\ \forall y\ \exists z\ R(A_m,A_n,x,y,z)$
is uniformly $\Pi^0_2$.
Thus, we can define chip functions $c_{m,n}$, uniformly for all $m<n$, that award infinitely many chips to $(m,n,k)$
just if $\exists x\le k\ \forall y\ \exists z\ R(A_m,A_n,x,y,z)$, i.e., if there is some $x\le k$ witnessing $A_m E A_n$.
(Saying that $c_{m,n}$ \emph{awards a chip to $(m,n,k)$ at stage $t$} means that $c_{m,n}(m,n,t) = k$.)
We will arrange these chip functions so that each $c_{m,n}$ has domain $\set{\la m,n,t\ra}{t\in\omega}$
and image $\omega$, with each element of $\omega$ lying in the domain of exactly one chip function.
Finally, for convenience, we define $c_{n,m}=c_{m,n}$ whenever $m<n$,
taking advantage of the symmetric nature of $E$.

We now index the set of all prime numbers as $\set{p_{m,n,k}}{(m,n,k)\in\omega^3~\&~m<n}$.
In our construction, we aim to achieve the following:
\begin{itemize}
\item If $(m,n,k)$ receives infinitely many chips, then $1$ is infinitely divisible
by $p_{m,n,k}$ in both $G_m$ and $G_n$.
\item if $(m,n,k)$ receives only finitely many chips, then 
$$(\exists r\in\omega)
\left[\frac1{p_{m,n,k}^{r-1}}\in G_m~\&~\frac1{p_{m,n,k}^r}\notin G_m~\&~\frac1{p_{m,n,k}^r}\in G_n~\&~\frac1{p_{m,n,k}^{r+1}}\notin G_n\right].$$
\end{itemize}

We give the procedure for each triple $(m,n,k)$ individually,
as there is no interaction between any two such procedures, although the
procedure here does involve other chip functions besides $c_{m,n}$.
Fix $(m,n,k)$ with $m<n$, and write $p=p_{m,n,k}$ for simplicity,
and fix $N=n+k$.  Every group $G_l$ contains $1$, hence contains $\Z$.
The entire purpose of this procedure is to determine,
for every group $G_l$ (not just $G_m$ and $G_n$), which negative
powers of $p$ lie in $G_l$.  We will do this in such a way that, if $c_{m,n}$ awards finitely many chips to $(m,n,k)$, then some power $p^{-r}$ with $r>0$ will lie in $G_n$ but not in $G_m$;
moreover, every $G_l$ will contain $p^{-(r-1)}$ and none will contain $p^{-(r+1)}$.
On the other hand, if $c_{m,n}$ awards infinitely many chips to $(m,n,k)$, then every $G_l$ will contain
every negative power of $p$.  Thus this procedure makes sure that $G_m$ will be isomorphic to $G_n$ if and only if there is some $x$ witnessing $A_m E A_n$, but it will also
do right by those $G_l$ with $l\leq N$ and $l\notin\{m,n\}$, as described further down.

At stage $0$, we define $p^{-1}$ to lie in $G_{n,0}$.  No negative power of $p$
lies in any other $G_{l,0}$, including $l=m$.  (Every $G_{l,0}$ contains $1$, however.)

A stage $s+1$ with $c_{m,n}(s)=k$ is called a \emph{chip stage} for $p_{m,n,k}$.
There is always a (least) $r>0$ with $p^{-r}\notin G_{m,s}$:  this $r$ is the \emph{key exponent}
for $p_{m,n,k}$ at stage $s$.  Every $G_{l,s}$ will contain $p^{-(r-1)}$.
The power $p^{-r}$ will lie in $G_{n,s}$ and many other $G_{l,s}$, 
but not $G_{m,s}$. No $G_{l,s}$
will contain $p^{-(r+1)}$.  We define the following linear order on the numbers $\leq N$:
$$ m \prec n \prec 0\prec 1\prec\cdots\prec m-1\prec m+1\prec\cdots\prec n-1\prec n+1\prec\cdots\prec N-1\prec N$$
and follow these instructions at stage $s+1$:
\begin{itemize}
\item
$p^{-r}$ enters every $G_{l,s+1}$, and $p^{-(r+1)}$ enters $G_{n,s+1}$ but not $G_{m,s+1}$.
\item
For every $l>N$, $p^{-(r+1)}$ enters $G_{l,s+1}$.
\item
For each $l\leq N$ with $l\neq m$ and $l\neq n$, we find the greatest $t_l\leq s$
(if any exists) such that $(\exists j\prec l)c_{j,l}(t_l)\leq N$.  We then proceed through the $\prec$ ordering.
Already $p^{-(r+1)}\in G_{n,s+1}$ but $\notin G_{m,s+1}$.  For each subsequent $l$
in turn (under $\prec$), enumerate $p^{-(r+1)}$ into $G_{l,s+1}$ if and only if $p^{-(r+1)}$ was already
enumerated into $G_{j,s+1}$, where $c_{j,l}(t_l)\leq N$.  (If it exists, this $j$ is unique,
because $t_l$ lies in the domain of only one chip function $c_{j,l}$.)  If there was
no such stage $t_l$, leave $p^{-(r+1)}$ out of $G_{l,s+1}$.
\end{itemize}
Thus every $G_{l,s+1}$ now contains $p^{-r}$, but none contains $p^{-(r+2)}$.
At the next stage, $(r+1)$ will have replaced $r$ as the key exponent for this $p$.
This completes the instructions when $c_{m,n}(s)=k$ --- but our procedure also has instructions
to follow at the \emph{non-chip stages}, i.e, those $s+1$ such that $c_{m,n}(s)\neq k$.

At these non-chip stages $s+1$, again we have a key exponent $r>0$ for $p$
at stage $s$, with $p^{-r}\notin G_{m,s}$.
For this $r$, $p^{-r}$ lies in $G_{n,s}$ and also in many other $G_{l,s}$, but none of these
contains $p^{-(r+1)}$.  We use the same order $\prec$ as in the preceding paragraph,
define $t_l$ the same way for each $l\leq N$ except for $m$ and $n$ themselves,
and check whether
$$(\exists j\leq N)(\exists l\leq N)~[j\prec l~\&~l\neq n~\&~
t_l\in\dom{c_{j,l}}~\&~((p^{-r}\notin G_{j,s}~\&~p^{-r}\in G_{l,s})\vee(p^{-r}\notin G_{l,s}~\&~p^{-r}\in G_{j,s}) )].$$
(Notice that neither $m$ nor $n$ can serve as $l$ here, since $j\prec l$ and $l\neq n$.)
If there are no such $j$ and $l$, then we change nothing at this stage, because every
$G_{l,s}$ already is equal to the preceding $G_{j,s}$ for which it seems most likely that
$A_l E A_j$.  However, if such $j$ and $l$ exist, then we repeat the procedure from above:
\begin{itemize}
\item
$p^{-r}$ enters every $G_{l,s+1}$, and $p^{-(r+1)}$ enters $G_{n,s+1}$ but not $G_{m,s+1}$.
\item
For every $l>N$, $p^{-(r+1)}$ enters $G_{l,s+1}$.
\item
For each $l\leq N$ with $l\neq m$ and $l\neq n$, we
proceed through the $\prec$ ordering.
Already $p^{-(r+1)}\in G_{n,s+1}$ but $\notin G_{m,s+1}$.  For each subsequent $l$
in turn, enumerate $p^{-(r+1)}$ into $G_{l,s+1}$ if and only if $p^{-(r+1)}$ was already
enumerated into $G_{j,s+1}$, where $c_{j,l}(t_l)\leq N$.  
If there was no such stage $t_l$, leave $p^{-(r+1)}$ out of $G_{l,s+1}$.
\end{itemize}
In this case, every $G_{l,s+1}$ now contains $p^{-r}$, but none contains $p^{-(r+2)}$.
This completes this stage, and the construction.

Of course, each $G_l$ is the additive subgroup of $\Q$ generated by $\cup_s G_{l,s}$,
understanding that each $G_{l,s}$ will contain powers of many different primes
$p_{m,n,k}$, since we run the construction above for all triples $(m,n,k)$ with $m<n$.
We now prove the relevant facts about the construction.
\begin{lemma}
\label{lemma:iso}
Fix any $(m,n,k)$.  If $c_{m,n}^{-1}(k)$ is infinite (so $k$ received infinitely many chips
from $c_{m,n}$), then $1$ is divisible by every power of $p_{m,n,k}$ in every $G_l$.
\end{lemma}
We will write ``$p_{m,n,k}^{-\infty}\in G_l$'' to denote that every power $p_{m,n,k}^{-r}$
lies in $G_l$.  Of course this is just shorthand:  there is no actual element $p_{m,n,k}^{-\infty}$
(and there is no finite stage $s$ by which all of these powers have entered $G_{l,s}$).
\begin{pf}
Every time $k$ received a chip from $c_{m,n}$, we adjoined new powers of $p_{m,n,k}$
to both $G_m$ and $G_n$, and ensured that the power in $G_m$ also lies in every $G_l$ at that stage.\qed
\end{pf}

Notice, however, that even if $k$ received only finitely many chips from $c_{m,n}$,
it is still conceivable that all powers of $p_{m,n,k}$ lie in every $G_l$ because the second part of the procedure was activated infinitely often.  This could occur if there
exist $l\in\omega$ and $k_0, k_1<N=n+k$ such that $c_{m,l}^{-1}(k_0)$
and $c_{n,l}^{-1}(k_1)$ are both infinite.  In this case, the second part of the procedure
may have been forced to increase the power $r$ with $p^{-r}\in G_{m,s+1}$ at infinitely
many stages $s$, as a new chip for $k_0$ from $c_{m,l}$ appeared, followed by a new chip
for $k_1$ from $c_{n,l}$, and then $k_0$ again, and so on.  However, in this case,
$k_0$ and $k_1$ guarantee (respectively) that $A_m E A_l$ and $A_l E A_n$,
so in fact $A_m E A_n$ in this case.  So, while $k$ itself may have received only
finitely many chips from $c_{m,n}$, some other $k'>k$ must have received infinitely many.
We now state this possibility formally.
\begin{lemma}
\label{lemma:inf}
The following are equivalent, for each $m<n$ and each $k$.
\begin{enumerate}
\item
For some $l$, $p_{m,n,k}^{-\infty}\in G_l$.
\item
For every $l$, $p_{m,n,k}^{-\infty}\in G_l$.
\end{enumerate}
Moreover, if these hold, then $A_m E A_n$.
\end{lemma}
\begin{pf}
(2) trivially implies (1), and (1)$\implies$(2) is clear from the construction for $(m,n,k)$:
at every stage $s$, $G_{n,s}$ is ``ahead of'' $G_{m,s}$ by one power of $p=p_{m,n,k}$
(meaning that there is some $r>0$ with $p^{-r}\in G_{n,s}-G_{m,s}$, while
$p^{-(r-1)}\in G_{m,s}$ and $p^{-(r+1)}\notin G_{n,s}$).  Moreover, for every other $l$,
$G_{l,s}$ is either ``behind'' with respect to powers of this $p$, i.e.,
$G_{l,s}\cap\set{p^q}{q>0}=G_{m,s}\cap\set{p^q}{q>0}$, or else ``ahead''
with respect to those powers, i.e.,
$G_{l,s}\cap\set{p^q}{q>0}=G_{n,s}\cap\set{p^q}{q>0}$.  So, if any $l$ at all
has $p_{m,n,k}^{-\infty}\in G_l$, then so do $G_n$ and $G_m$ and every other $G_j$.

Suppose now that $A_m\centernot{E}A_n$.
Say that a chip $c_{m,n}(s)$ is \emph{false} if $c_{m,n}^{-1}(c_{m,n}(s))$ is finite.
(That is, this chip has the potential to mislead us.)  Fixing $m$, $n$, $k$, and $N=n+k$,
we see that there is some stage $s_0$ such that no $c_{j,l}$ with $j < l \leq N$
ever gives a false chip $c_{j,l}(s)$ to some $k' \le N$ at any stage $s\geq s_0$, because there can only
be finitely many such false chips given at all. 
Now let $s_1>s_0$ be a stage such that,
for each $j<l \leq N$ and each $i\leq N$ with $c_{j,l}^{-1}(i)$ infinite, there is some
$s$ between $s_0$ and $s_1$ with $c_{j,l}(s)=i$.  Let $s_2>s_1$ be so large that
this has happened again between $s_1$ and $s_2$, and define $s_3<s_4<\cdots$
similarly.  (One might say that the chip functions have dealt out a full round of true chips
between each $s_q$ and $s_{q+1}$.)

Since $A_m\centernot E A_n$, we know that $c_{m,n}(s)>N$ (if $c_{m,n}(s)$ is defined)
for every $s\geq s_0$.  Thus all subsequent stages in the procedure for $(m,n,k)$
are non-chip stages.  We consider $l=0$ first (assuming $m\neq 0$).  If any
$c_{0,m}^{-1}(i)$ with $i\leq N$ is infinite, then $c_{0,m}(s)=i$ for some $s$
between $s_0$ and $s_1$.  But then $A_0E A_m$, so $A_0\centernot E A_n$.
It follows that every $i'$ is a false chip for $c_{0,n}$, and thus $c_{0,n}(s)>N$ for every $s\geq s_0$.
Therefore, by stage $s_1$ we will have $p_{m,n,k}^{-(r+1)}\notin G_{m,s_1}\cup G_{0,s_1}$,
according to the construction at the non-chip stage $s+1$ at which $c_{0,m}(s)=i$,
and $G_{0,s}$ will stay even with $G_{m,s}$ forever after (in terms of powers of $p_{m,n,k}$),
because $c_{0,n}(s)>N$ for every $s\geq s_0$, as shown above.

A similar argument shows that if any $c_{0,n}^{-1}(i)$ with $i\leq N$ is infinite,
then $p_{m,n,k}^{-(r+1)}\in G_{n,s_1}\cap G_{0,s_1}$, and $G_{0,s}$ will stay
even with $G_{n,s}$ forever after (in terms of powers of $p_{m,n,k}$).  And if
no $c_{m,0}^{-1}(i)$ nor any $c_{0,n}^{-1}(i)$ with $i\leq N$ is infinite, then
both $c_{0,m}(s)>N$ and $c_{0,n}(s)>N$ for every $s\geq s_0$, in which case
either $G_{0,s}$ will stay even forever with $G_{m,s}$ (in case $G_{m,s_0}=G_{0,s_0}$),
or else it will stay even forever with $G_{n,s}$ (since then $G_{n,s_0}=G_{0,s_0}$).

Finally, notice that if $m=0$, then this same argument would hold with $1$ in place of $0$
(or with $2$ in place of $0$, in case $n=1$).  It really shows that by stage $s_1$, the next element $j_0$
after $n$ in the $\prec$-order must have linked its $G_{j_0}$ either to $G_m$ or to $G_n$ permanently
(as far as powers of $p_{m,n,k}$ are concerned).

But now the same argument applies to the subsequent element $j_1$ in the $\prec$-order,
between the stages $s_1$ and $s_2$.  Once $j_0$ has ``settled down'' this way, $j_1$
will either select (by stage $s_2$) which of $m$, $n$ and $j_0$ to link to, or else it
will never link to any of them but will keep the same position that it holds
at stage $s_1$.  In any case, $j_1$ never changes its position after stage $s_2$.
Continuing by induction, we see that after stage $s_N$, the final element $j_{N-1}$
in the $\prec$-order $m\prec n\prec j_0\prec\cdots\prec j_{N-1}$ on $\{0,1,\ldots,N\}$
will have never change its position again.  Thus, from stage $s_N$ on, no new powers
of $p_{m,n,k}$ are ever added to $G_m$ or $G_n$, or to any other $G_l$.
This proves the final claim of the lemma.
\qed\end{pf}

\begin{lemma}
\label{lemma:noniso}
If $A_m \centernot E A_n$, then $G_m\not\cong G_n$.
\end{lemma}
\begin{pf}
By hypothesis, for every $k\in\omega$, $c_{m,n}^{-1}(k)$ is finite.
Lemma \ref{lemma:inf} shows that there are only finitely many powers
of each $p_{m,n,k}$ in each $G_l$, so (for a single fixed $k$)
let $r$ be maximal with $p_{m,n,k}^{-r}\in G_m$,
and fix the stage $s+1$ at which $p_{m,n,k}^{-r}$ was adjoined to $G_{m,s+1}$.
By the construction, we have $p_{m,n,k}^{-(r+1)}\in G_{n,s+1}$.  Thus,
for every~$k$, $1$ is divisible by $p_{m,n,k}^{(r+1)}$ in $G_n$
(for the $r$ corresponding to this $k$), but not in $G_m$.
It follows that $G_m\not\cong G_n$.
\qed\end{pf}

\begin{lemma}
\label{lemma:final}
If $A_m E A_n$, then $G_m\cong G_n$.
\end{lemma}
\begin{pf}
For every $m'<n'$ and every $i$, Lemma \ref{lemma:iso} shows that
$$ p_{m',n',i}^{-\infty}\in G_m \iff p_{m',n',i}^{-\infty}\in G_n.$$
In particular, if $k$ is the least such that $c_{m,n}^{-1}(k)$ is infinite,
then for every $i\geq k$, $p_{m,n,i}^{-\infty}$ lies in both $G_m$ and $G_n$.
For the remaining finitely many $i$, there will be powers $r$ (likely distinct for different $i$)
such that $p_{m,n,i}^{-r}$ lies in $G_n$ but not $G_m$ (assuming without loss of generality
that $m<n$).  Recall that, to show $G_m\cong G_n$, we need to show that there are only
finitely many prime powers $p^{-r}$ that lie in one of $G_m$ and $G_n$
but not in the other, so these finitely many values $i<k$ do not upset us.
(The initial result made it clear that $p_{m,n,i}^{-\infty}$ lies in neither $G_m$ nor $G_n$,
so these $i$ really do yield only finitely many such prime powers.)

However, there are many more primes to be considered.  We claim that
for those primes $p_{m',n',k'}$ with $(m',n')\neq(m,n)$ and $n'+k'=N'\geq \max(n,k)$,
each power $p_{m',n',k'}^r$ will lie in $G_m$ if and only if it lies in $G_n$. 
Recall our convention that $m'<n'$ in all these triples, so there
are only finitely many triples $(m',n',k')$ with $n'+k'<\max(n,k)$.
Consequently, this claim, once proven, will suffice to show that $G_m\cong G_n$.

The claim holds because there are infinitely many $s$ with $c_{m,n}(s)=k$.
Each such stage $s+1$ is a non-chip stage in the procedure for $(m',n',k')$,
because $\dom{c_{m,n}}\cap\dom{c_{m',n'}}=\emptyset$ when $(m,n)\neq(m',n')$.
With $m<n\leq N'$ and $k\leq N'$, the procedure for $(m',n',k')$ at 
stage $s+1$ will set $t_n=s$, the greatest stage $\leq s$ at which
some $j\prec n$ (namely $m$) has $c_{j,n}(s)\leq N'$.
Using the key exponent $r$ at this stage, this procedure will ask whether
$$(p_{m',n',k'}^{-r}\in G_{m,s} \iff p_{m',n',k'}^{-r}\in G_{n,s}).$$
If not, it will increase the key power by $1$, to $r+1$, and ensure that 
$$(p_{m',n',k'}^{-(r+1)}\in G_{m,s} \iff p_{m',n',k'}^{-(r+1)}\in G_{n,s}).$$
(Whether this power is in both these sets or out of them both depends on $t_m$;
this is irrelevant to our argument here.)
If $G_{m,s}$ and $G_{n,s}$ were ``even'' with respect to this prime,
it is still possible that the procedure will increase the key power on account
of $t_l$ for some other $l\leq N'$, but even if it does so, it will still keep
$$(p_{m',n',k'}^{-(r+1)}\in G_{m,s} \iff p_{m',n',k'}^{-(r+1)}\in G_{n,s}).$$
Since this holds at infinitely many stages $s$ (namely, those in $c_{m,n}^{-1}(k)$),
it is clear that for every power $r$, $p_{m',n',k'}^{-r}$ lies in $G_m$
if and only if it lies in $G_n$.  As noted above, this completes the proof.
\qed\end{pf}
By Lemma \ref{lemma:noniso} and \ref{lemma:final}, our procedure $\Phi$
does indeed compute a countable reduction from $E$ to $\TFAb_1$.
\qed\end{pf}

\begin{cor}
Uniformly for each $d>0$,
there is a computable countable reduction from the space $\TD_d$ of fields
of transcendence degree $d$ over $\Q$ to the space $\TFAb_1$ of torsion-free
abelian groups of rank $1$.
\end{cor}

\begin{proof}
Note that the isomorphism problem on $\TD_d$ is uniformly $\Sigma^0_3$ via
\begin{align*}
F \cong E \Leftrightarrow \exists \bar a\in F, \bar b \in E~[ & (\text{$\bar a$ is an algebraically independent $d$-tuple})\\
&\wedge(\text{$\bar b$ is an algebraically independent $d$-tuple})\\
&\wedge(\text{$\bar a\mapsto \bar b$ extends to an isomorphism $F \to E$})].
\end{align*}
Thus, by the previous Theorem \ref{tfab-sigma3}, there is a uniform computable countable reduction from $\TD_d$ to $\tfa_1$. 
\end{proof}

Since the isomorphism relation on $\TFAb_r$ is similarly $\Sigma^0_3$, uniformly
in $r$, we have a similar corollary for groups, which contrasts with Theorem
\ref{thm:HjorthThomas} of Hjorth and Thomas.

\begin{cor}
Uniformly for each $r\geq 0$,
there is a computable countable reduction from the space $\TFAb_r$ of torsion-free
abelian groups of rank $r$ to the space $\TFAb_1$.
\end{cor}



\end{document}